\documentclass[11pt]{article}

\usepackage{amsmath, amsthm, amssymb,fullpage,enumerate,tikz,multirow}
\usepackage{hyperref}
\usepackage{colortbl}
\usepackage{arydshln}
\usepackage{hhline}

\newtheorem{thm}{Theorem}[section]
\newtheorem{lem}[thm]{Lemma}
\newtheorem{prop}[thm]{Proposition}
\newtheorem{conj}[thm]{Conjecture}
\newtheorem{cor}[thm]{Corollary}

\theoremstyle{definition}
\newtheorem{defn}[thm]{Definition}
\newtheorem{ex}[thm]{Example}

\theoremstyle{remark}

\usepackage{lipsum}
\usepackage{tikz}
\usepackage[normalem]{ulem}

\usepackage[isbn=false,url=false,doi=false,backend=bibtex]{biblatex}
\newcommand{\mfS}{\mathfrak{S}}
\newcommand{\mcO}{\mathcal{O}}

\newcommand{\st}{\operatorname{st}}
\newcommand{\emm}{\operatorname{em}}
\newcommand{\Em}{\operatorname{Em}}
\newcommand{\card}[1]{{\lvert #1 \rvert}} 	% cardinality
\def\we{\sim}
\def\swe{\overset{s}{\we}}
\def\sswe{\overset{ss}{\we}}

\title{Wilf equivalence relations for consecutive patterns}

\author{Tim Dwyer\thanks{Department of Mathematics, Dartmouth College, Hanover, NH, USA.} \and Sergi Elizalde\thanks{Department of Mathematics, Dartmouth College, Hanover, NH, USA. E-mail: \texttt{sergi.elizalde@dartmouth.edu}.
Partially supported by Simons Foundation grant \#280575.}}

\date{}

\begin{document}

\maketitle

\begin{abstract}
Two permutations $\pi$ and $\tau$ are {\em c-Wilf equivalent} if, for each $n$, the number of permutations in $\mfS_n$ avoiding $\pi$ as a consecutive pattern (i.e., in adjacent positions) is the same as the number of those avoiding~$\tau$.
In addition, $\pi$ and $\tau$ are {\em strongly c-Wilf equivalent} if, for each $n$ and~$k$, the number of permutations in $\mfS_n$ containing $k$ occurrences of $\pi$ as a consecutive pattern is the same as for~$\tau$. 
In this paper we introduce a third, more restrictive equivalence relation, defining $\pi$ and $\tau$ to be {\em super-strongly c-Wilf equivalent} if the above condition holds for any set of prescribed positions for the $k$ occurrences.
We show that, when restricted to non-overlapping permutations, these three equivalence relations coincide.

We also give a necessary condition for two permutations to be strongly c-Wilf equivalent. Specifically, we show that if \(\pi,\tau\in\mfS_m\)
are strongly c-Wilf equivalent, then \(|\pi_m-\pi_1|=|\tau_m-\tau_1|\). In the special case of non-overlapping permutations $\pi$ and $\tau$, 
this proves a weaker version of a conjecture of the second author stating that $\pi$ and $\tau$ are c-Wilf equivalent if and only if \(\pi_1=\tau_1\) and \(\pi_m=\tau_m\), up to trivial symmetries. 
Finally, we strengthen a recent result of Nakamura and Khoroshkin--Shapiro giving sufficient conditions for strong c-Wilf equivalence.
\end{abstract}

\section{Introduction and summary of results}

Inspired by the work of Knuth~\cite{Knuth1969}, the last three decades have seen an explosion of research in permutation patterns. Aside from the study of classical patterns, a number of questions have arisen involving different types of patterns in permutations, including \emph{consecutive, vincular, bivinvular, mesh} and \emph{barred} patterns. A common question in all of these settings is, for a given pattern \(\pi\) of length \(m\), how many permutations \(\sigma\) of length \(n\) avoid this pattern. This is a very difficult question in general. Another related question is when two patterns have the same number of permutations of length \(n\) avoiding them, for all $n$. In the classical case, two patterns with this property are said to be Wilf equivalent. The classification of patterns into Wilf equivalence classes is a wide open problem; see~\cite{MR2290807,MR1900628,MR1297387} for some results in this area. 

In this paper we focus on the analogous question for consecutive patterns, that is, patterns that occur in adjacent positions of the permutation. In this case, the notion analogous to Wilf equivalence is called {\em c-Wilf equivalence}, following the terminology from~\cite{NakamuraClusters}. Even though the classification of patterns into c-Wilf equivalence classes is also open, in this paper we give a natural necessary condition for two patterns to be c-Wilf equivalent. We also investigate the related notions of strong and super-strong c-Wilf equivalence.

Consecutive patterns appear naturally when defining permutation statistics such as descents, peaks, valleys and runs, and also when defining alternating permutations. The systematic enumeration of permutations avoiding consecutive patterns started in~\cite{EN2003}, and it is now an active area of research (see the survey~\cite{Elizalde2016}). 

\medskip

Let \(\mfS_n\) be the symmetric group on \([n]\), and let \(\mfS=\bigcup_{n\ge0}\mfS_n\). For $\sigma\in\mfS_n$, we write $\sigma=\sigma_1\sigma_2\dots\sigma_n$ and let $|\sigma|=n$ denote its length.
Given two permutations \(\pi\in\mfS_m\) and \(\sigma\in\mfS_n\), we say that \(\sigma\) \emph{contains \(\pi\) as a consecutive pattern} if there is an \(i\in[n-m+1]\) for which \(\st(\sigma_{i}\dots\sigma_{i+m-1})=\pi\), where $\st$ is the {\em standardization} operation that 
replaces the smallest entry with a \(1\), the next smallest with a \(2\) and so on. 
The substring $\sigma_{i}\dots\sigma_{i+m-1}$ is called an \emph{occurrence} or an \emph{embedding} of \(\pi\) in~\(\sigma\), and we say it occurs at position $i$. For example, the permutation \(\sigma=43815672\) contains the consecutive pattern \(51234\) at position \(3\), since $\st(81567)=51234$.  Define 
$$\Em(\pi,\sigma)=\{i:\st(\sigma_i\ldots\sigma_{i+m-1})=\pi\}$$ 
to be the set of positions of occurrences of \(\pi\) in \(\sigma\), and let \(\emm(\pi,\sigma)=\card{\Em(\pi,\sigma)}\). For example,    $\Em(231,245361)=\{2,4\}$ and $\emm(231,245361)=2$. Note that $\Em(21,\sigma)$ is just the descent set of~$\sigma$.
We will indistinctively use the words {\em permutation} and {\em pattern} to refer to $\pi\in\mfS_m$.

To count occurrences of a consecutive pattern $\pi$ in permutations, we use the exponential generating function
$$F_\pi(u,z)=\sum_{\sigma\in \mfS}u^{\emm(\pi,\sigma)}\frac{z^{|\sigma|}}{|\sigma|!}=\sum_{n,k\ge0} a_{n,k}^\pi\, u^k\frac{z^n}{n!},$$
where \(a_{n,k}^\pi\) is the number of permutations \(\sigma\in\mfS_n\) with \(\emm(\pi,\sigma)=k\). Explicit formulas for $F_\pi(u,z)$ are known for a few specific patterns $\pi$~\cite{EN2003,ElizaldeNoyClusters2012}.
However, finding expressions for $F_{\pi}(u,z)$ in general is a difficult problem.

Instead, in this paper we focus on some natural equivalence relations that arise from the definition of consecutive patterns. 

\begin{defn}
Two permutations $\pi$ and $\tau$ are \emph{c-Wilf equivalent}, denoted $\pi\we\tau$, if \[F_\pi(0,z)=F_\tau(0,z),\] and they are
\emph{strongly c-Wilf equivalent}, denoted $\pi\swe\tau$, if \[F_\pi(u,z)=F_\tau(u,z).\] 
\end{defn}

Equivalently, $\pi\we\tau$ if $a_{n,0}^\pi=a_{n,0}^\tau$ for all $n$, and $\pi\swe\tau$ if $a_{n,k}^\pi=a_{n,k}^\tau$ for all $n$ and $k$. Clearly, strong c-Wilf equivalence implies c-Wilf equivalence.
It was conjectured by Nakamura~\cite{NakamuraClusters} that these relations are actually the same.

\begin{conj}[{\cite[Conjecture 5.6]{NakamuraClusters}}]\label{conj:Nak} Two permutations $\pi$ and $\tau$ are c-Wilf equivalent if and only if they are strongly c-Wilf equivalent.
\end{conj}

The analogue to Conjecture~\ref{conj:Nak} for classical patterns is false, already for patterns of length three.

Clearly, every permutation \(\pi\in\mfS_m\) is strongly c-Wilf equivalent to its reversal \(\pi^R=\pi_m\ldots\pi_1\), its complement \(\pi^C=(m+1-\pi_1)\ldots (m+1-\pi_m)\), and its reverse-complement \(\pi^{RC}=(m+1-\pi_m)\ldots(m+1-\pi_1)\). The smallest example of a c-Wilf equivalence that does not arise from these symmetries is given by $1342\swe1432$, as shown in~\cite{EN2003}.

For \(\pi\in\mfS_m\), its \emph{overlap set} \(\mcO_\pi\) is defined as the set of indices \(i\in[m-1]\) such that \(\st(\pi_{i+1}\ldots \pi_{m})=\st(\pi_1\ldots\pi_{m-i})\). The overlap set keeps track of which suffixes and prefixes of \(\pi\) have the same standardization. Note that we always have $m-1\in\mcO_\pi$, since $\st(\pi_m)=\st(\pi_1)=1$. The permutations in $\mfS_m$ for which \(\mcO_\pi=\{m-1\}\)
are called \emph{non-overlapping} (or sometimes \emph{minimally overlapping}). For example, $\pi=16358472$ is non-overlapping, since $\mcO_\pi=\{7\}$. On the other hand, the overlap set of $2143$ is $\{2,3\}$. It was shown by B\'ona~\cite{MR2924741} that the fraction of non-overlapping permutations in $\mfS_m$ is about $0.364$ in the limit as $m\to\infty$.
Conjecture~\ref{conj:Nak} was proved in \cite[Lem.~3.2]{ElizaldeMostLeast2013} and \cite[Thm.~11]{MendesRemmel} in the special case of non-overlapping  permutations:
\begin{lem}[\cite{ElizaldeMostLeast2013,MendesRemmel}]\label{lem:eq_seq}
Let \(\pi,\tau\in\mfS_m\) be non-overlapping. If $\pi\we\tau$, then $\pi\swe\tau$.
\end{lem}
A sufficient condition for strong c-Wilf equivalence of two permutations with the same overlap set was given independently by
Khoroshkin and Shapiro \cite{khoshap}, and Nakamura \cite{NakamuraClusters}.

\begin{thm}[\cite{khoshap,NakamuraClusters}]\label{thm:KS}
Let \(\pi,\tau\in\mfS_m\). If \(\mcO_\pi=\mcO_\tau\) and, for all \(i\in\mcO_\pi\), we have
\begin{equation}
\label{eq:KS}
\{\pi_1,\dots,\pi_{m-i}\}=\{\tau_1,\dots,\tau_{m-i}\} \quad \mbox{and} \quad \{\pi_{i+1}\dots,\pi_m\}=\{\tau_{i+1},\dots,\tau_m\}
\end{equation}
then $\pi\swe\tau$.
\end{thm}

\begin{ex}\label{ex:s} Every permutation $\pi$ from the list
$$1734526, 1735426, 1743526, 1745326, 1753426, 1754326$$ satisfies $\mcO_\pi=\{6,7\}$, $\{\pi_1,\pi_2\}=\{1,7\}$  and $\{\pi_6,\pi_7\}=\{2,6\}$. It follows that all the permutations on this list are strongly c-Wilf equivalent.
\end{ex}

In the special case of non-overlapping permutations $\pi$ and $\tau$, Theorem~\ref{thm:KS} simply states that if \(\pi_1=\tau_1\) and \(\pi_m=\tau_m\), then $\pi\swe\tau$. This fact had been shown in~\cite{DotKho2013,DR2011}.
A converse of this statement for non-overlapping permutations has been conjectured in \cite{ElizaldeMostLeast2013}. To state the conjecture, first define $\pi\in\mfS_m$ to be in \emph{standard form} if \(\pi_1<\pi_m\) and \(\pi_1+\pi_m\le m+1\). Note that, for any  $\pi\in\mfS_m$, at least one permutation among $\pi,\pi^R,\pi^C,\pi^{RC}$ is in standard form. 

\begin{conj}[\cite{ElizaldeMostLeast2013}]
\label{elizaldeconj}
Let \(\pi,\tau\in\mfS_m\) be non-overlapping and in standard form. If $\pi\we\tau$, then \(\pi_1=\tau_1\) and \(\pi_m=\tau_m\).
\end{conj}

Since the condition \(\pi_1=\tau_1\) and \(\pi_m=\tau_m\) is sufficient for non-overlapping permutations $\pi,\tau\in\mfS_m$ to be 
strongly c-Wilf equivalent, the above conjecture would completely characterize (strong) c-Wilf equivalence classes for non-overlapping patterns.

Even though Conjecture~\ref{elizaldeconj} applies only to non-overlapping patterns, we can formulate a related conjecture without this restriction. As mentioned above, for non-overlapping patterns, c-Wilf equivalence is the same as strong c-Wilf equivalence, so the following conjecture includes Conjecture~\ref{elizaldeconj} as a special case.

\begin{conj}
\label{dwyerconj}
Let \(\pi,\tau\in\mfS_m\) be in standard form. If $\pi\swe\tau$, then \(\pi_1=\tau_1\) and \(\pi_m=\tau_m\).
\end{conj}

It is natural to ask if, even more generally, the converse of Theorem~\ref{thm:KS} holds for permutations in standard form, that is, whether any two strongly c-Wilf equivalent permutations in standard form always satisfy the hypotheses of this theorem.
While we prove in Corollary~\ref{cor:sscoverlap} that the first part of the hypothesis is always satisfied, Equation~\eqref{eq:KS} does not hold in general. For example, $\pi=123546$ and $\tau=124536$ are strongly c-Wilf equivalent, as shown in~\cite{ElizaldeNoyClusters2012}.
However, $4\in\mcO_\pi=\mcO_\tau$ but $\{\pi_5,\pi_6\}=\{4,6\}\neq\{3,6\}=\{\tau_5,\tau_6\}$.

Section~\ref{sec:clusterposets} gives some background on the cluster method of Goulden and Jackson~\cite{ClusterOriginal}, as well as an interpretation of certain coefficients as counting linear extensions of posets~\cite{ElizaldeNoyClusters2012}. These posets will be a key tool in many of our proofs. In particular, analyzing their structure in the case of non-overlapping patterns will lead to the proof of the following result, which appears in Section~\ref{sec:proof1}.

\begin{thm}
\label{thm1}
Conjecture~\ref{elizaldeconj} implies Conjecture~\ref{dwyerconj}.
\end{thm}

The above theorem states that if the conjecture from~\cite{ElizaldeMostLeast2013} about non-overlapping patterns holds, then so does our more general conjecture about arbitrary patterns, and thus these two conjectures are equivalent.

Even though these conjectures remain open, we are able to prove in Section \ref{sec:proof2} that the following weaker version of Conjecture \ref{dwyerconj} holds:
\begin{thm}
\label{thm2}
Let \(\pi,\tau\in\mfS_m\) be in standard form. If $\pi\swe\tau$, then $$\pi_m-\pi_1=\tau_m-\tau_1.$$
\end{thm}

In Section~\ref{superstrong} we introduce a third equivalence relation on permutations that refines strong c-Wilf equivalence.
Given a set of positive integers $S$, define \(a_{n,S}^\pi\) to be the number of permutations \(\sigma\in\mfS_n\) with \(\Em(\pi,\sigma)=S\). 

\begin{defn}
Two permutations $\pi$ and $\tau$ are \emph{super-strongly c-Wilf equivalent}, denoted $\pi\sswe\tau$, if $$a_{n,S}^\pi=a_{n,S}^\tau$$ for all \(n\) and \(S\).
\end{defn}

Clearly, super-strong c-Wilf equivalence implies strong c-Wilf equivalence. It is immediate that $\pi\sswe\pi^C$ for all $\pi$, but we have $\pi\not\sswe\pi^R$ in general.
In Section~\ref{superstrong} we prove the following generalization of Theorem \ref{thm:KS}. 

\begin{thm}\label{thm:DE-ss}
Let \(\pi,\tau\in\mfS_m\). If \(\mcO_\pi=\mcO_\tau\) and, for all \(i\in\mcO_\pi\), we have
$$\{\pi_1,\dots,\pi_{m-i}\}=\{\tau_1,\dots,\tau_{m-i}\} \quad \mbox{and} \quad \{\pi_{i+1}\dots,\pi_m\}=\{\tau_{i+1},\dots,\tau_m\},$$ 
then $\pi\sswe\tau$. 
\end{thm}

\begin{ex}
By Theorem~\ref{thm:DE-ss}, the six permutations in Example~\ref{ex:s} are in fact super-strongly c-Wilf equivalent. Another example of an application of this theorem is Example~\ref{ex:ss}.
\end{ex}

The proof of Theorem~\ref{thm:DE-ss} is based on an extension of the cluster method, which allows us to keep track not only of the number of occurrences of a pattern but also of their positions, as stated in Proposition~\ref{prop:Cluster-ss}. In Theorem~\ref{thm:eq-sseq} we show that, for non-overlapping patterns, c-Wilf equivalence implies super-strong c-Wilf equivalence, generalizing Lemma~\ref{lem:eq_seq}. Finally, Theorem~\ref{thm:ssc-SufficientMAXMIN} describes some conditions under which $\pi\sswe\pi^R$, which complete the classification of patterns of length~5 into super-strong c-Wilf equivalence classes.

\section{The cluster method}
The cluster method was introduced by Goulden and Jackson \cite{ClusterOriginal,MR702512} in order to enumerate words over a given alphabet 
with respect to the number of occurrences of specific substrings. 
It has since been adapted to consecutive permutation patterns \cite{NakamuraClusters,ElizaldeNoyClusters2012,ElizaldeMostLeast2013} and to  the generalized factor order over the positive integers \cite{MR3466374}.

Given a pattern $\pi$, the idea is to consider ordered pairs \((\sigma,S)\) with \(|S|=k\) and \(S\subseteq\Em(\pi,\sigma)\). We call such an ordered pair a \emph{marked permutation}, and we consider the occurrences of $\pi$ in positions in $S$ to be marked. We represent marked occurrences by underlining them in \(\sigma\). For example, for $\pi=321$, the marked permutation \((432179865,\{1,2,7\})\) can be represented as \(\underline{4\underline{32}}\underline{1}79\underline{865}\). The generating function for all marked permutations \((\sigma,S)\) is $$\sum_{(\sigma,S)}\frac{z^{|\sigma|}}{|\sigma|!}t^{|S|}=F_\pi(1+t,z).$$

The cluster method expresses this generating function in terms of the generating function for a special type of marked permutations called \emph{clusters}. 
\begin{defn}\label{def:cluster}
Let \(\pi\in\mfS_m\). A marked permutation \((\sigma,S)\) with $\sigma\in\mfS_n$ is a {\em \(\pi\)-cluster} if \(S=\{i_1<\cdots<i_k\}\subseteq\Em(\pi,\sigma)\) satisfies the following conditions:
\begin{enumerate}[(a)]
\item $i_1=1$, \ $i_k=n-m+1$,
\item \(i_{j+1}-i_j\in\mcO_\pi\) for all \(j\in[k-1]\).
\end{enumerate}
\end{defn}
In other words, both $\sigma_1$ and $\sigma_n$ belong to a marked occurrence, and each marked occurrence overlaps the next one.
The previous example of a marked permutation is not a \(321\)-cluster, but \(\uline{6\uline{54\uline{}}}\uline{\uline{3}}\uline{21}\) is. 
Define the cluster generating function
$$R_\pi(t,z)=\sum_{(\sigma,S)}\frac{z^{|\sigma|}}{|\sigma|!}t^{|S|}=\sum_{n,k\ge0} r_{n,k}^\pi\, t^k\frac{z^n}{n!},$$
where now the first sum is taken over all \(\pi\)-clusters \((\sigma,S)\), and \(r_{n,k}^\pi\) is the number of \(\pi\)-clusters \((\sigma,S)\) where  \(\sigma\in\mfS_n\) and \(|S|=k\). The numbers \(r_{n,k}^\pi\) are called the \emph{cluster numbers} of \(\pi\). 

A marked permutation can be identified with a sequence consisting of unmarked single entries interspersed with strings of overlapping marked occurrences of $\pi$ that would be $\pi$-clusters if the underlying word was standardized. For example, the marked permutation \(\underline{4\underline{32}}\underline{1}79\underline{865}\) corresponds to the sequence  \(\underline{4\underline{32}}\underline{1},7,9,\underline{865}\). 
This identification, together with the substitution $u=1+t$, provides the following connection between the generating functions $F_\pi$ and $R_\pi$. 

\begin{thm}[\cite{ClusterOriginal,NakamuraClusters}]
\label{clustermethod}
For any permutation \(\pi\), we have
$$F_\pi(u,z)=\frac{1}{1-z-R_\pi(u-1,z)}.$$
\end{thm}
It follows immediately that $\pi\swe\tau$ if and only if \(r_{n,k}^\pi=r_{n,k}^\tau\) for all \(n\) and \(k\).

An interesting corollary of Theorem \ref{clustermethod} is that in order for two permutations to be strongly c-Wilf equivalent they must have the same overlap set, giving a partial converse to Theorem~\ref{thm:KS}.

\begin{cor}
\label{cor:sscoverlap}
Let $\pi,\tau\in \mfS_m$. If $\pi\swe\tau$, then \(\mcO_\pi=\mcO_\tau\). 
\end{cor}
\begin{proof}
Since $\pi\swe\tau$, they have the same cluster numbers by Theorem~\ref{clustermethod}, that is, \(r_{n,k}^\pi=r_{n,k}^\tau\) for all $n$ and $k$. By definition, \(i\in\mcO_\pi\) if and only if $\st(\pi_{i+1}\ldots\pi_m)=\st(\pi_1\ldots\pi_{m-i})$. This condition is equivalent to the existence of \(\sigma\in\mfS_{i+m}\) with \(\st(\sigma_1\ldots\sigma_m)=\pi=\st(\sigma_{i+1}\ldots\sigma_{i+m})\). The number of such $\sigma$ is  \(r_{i+m,2}^\pi\) by definition. Therefore, \(i\in\mcO_\pi\) if and only if \(r_{i+m,2}^\pi\neq0\). It follows that
$$i\in\mcO_\pi\iff r_{i+m,2}^\pi\neq0\iff r_{i+m,2}^\tau\neq0\iff i\in\mcO_\tau.$$
\end{proof}

\subsection{Cluster posets}\label{sec:clusterposets}
Elizalde and Noy \cite{ElizaldeNoyClusters2012} established a connection between cluster numbers and linear extensions of posets.

Fix \(\pi\in\mfS_m\). We can write $$r_{n,k}^\pi=\sum_{S} r_{n,S}^\pi,$$ 
where the sum is over all sets \(S\subseteq[n-m+1]\) with $|S|=k$ satisfying conditions (a) and (b) in Definition~\ref{def:cluster}, and
\(r_{n,S}^\pi\) is the number of \(\sigma\in\mfS_n\) such that \(S\subseteq\Em(\pi,\sigma)\).
The number \(r_{n,S}^\pi\), which counts $\pi$-clusters of the form $(\sigma,S)$, is called a \emph{refined cluster number}. If $S$ does not satisfy conditions (a) and (b) in Definition~\ref{def:cluster}, we define \(r_{n,S}^\pi=0\) for convenience.

For each $n$ and $S$ satisfy conditions (a) and (b) above, we define a poset \(P_{n,S}^\pi\) on the set \(\{\sigma_1,\ldots,\sigma_n\}\) generated by the order relationships forced by the fact that $\sigma_1\dots\sigma_n$ must contain occurrences of $\pi$ at each $i\in S$.
We call \(P_{n,S}^\pi\) a \emph{cluster poset}. By construction, linear extensions of this poset correspond to permutations \(\sigma\in\mfS_n\) such that \(S\subseteq\Em(\pi,\sigma)\), and so \(P_{n,S}^\pi\) has exactly \(r_{n,S}^\pi\) linear extensions. 

For example, if \(\pi=513624\), $n=12$ and \(S=\{1,4,7\}\), then \(r_{12,\{1,4,7\}}^{513624}\) is the number of permutations \(\sigma\in\mfS_{12}\) satisfying
\begin{align}
\label{example1}
\st(\sigma_1\sigma_2\sigma_3\sigma_4\sigma_5\sigma_6)=\st(\sigma_4\sigma_5\sigma_6\sigma_7\sigma_8\sigma_9)=\st(\sigma_7\sigma_8\sigma_9\sigma_{10}\sigma_{11}\sigma_{12})=513624.
\end{align}
Noting that \(\pi^{-1}=253614\), Equation \eqref{example1} is equivalent to the following 3 chains of inequalities:
\begin{align*}
&\sigma_{2}<\sigma_{5}<\sigma_{3}<\sigma_{6}<\sigma_{1}<\sigma_{4},\\
&\sigma_{5}<\sigma_{8}<\sigma_{6}<\sigma_{9}<\sigma_{4}<\sigma_{7},\\
&\sigma_{8}<\sigma_{11}<\sigma_{9}<\sigma_{12}<\sigma_{7}<\sigma_{10}.
\end{align*}
The cluster poset \(P_{12,\{1,4,7\}}^{513624}\) is defined by the transitive closure of these relations, and its Hasse diagram is given in Figure~\ref{fig1}. Note that this poset is well-defined because all the symbols which appear in multiple chains have the same ordering in each chain, as is guaranteed by the fact that $S$ satisfies condition (b) in Definition~\ref{def:cluster}.

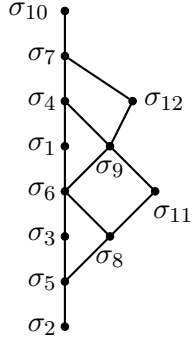
\begin{figure}[htb]
\centering
\begin{tikzpicture}[scale=1.2]
    \draw [thick] (0,0) -- (0,3.5);
    \draw[black,fill=black] (0,0) circle (.25ex);
    \draw[black,fill=black] (0,.5) circle (.25ex);
    \draw[black,fill=black] (0,1) circle (.25ex);
    \draw[black,fill=black] (0,1.5) circle (.25ex);
    \draw[black,fill=black] (0,2) circle (.25ex);
    \draw[black,fill=black] (0,2.5) circle (.25ex);
    \draw[black,fill=black] (0,3) circle (.25ex);
    \draw[black,fill=black] (0,3.5) circle (.25ex);

    \draw[black,fill=black] (.5,1) circle (.25ex);
    \draw[black,fill=black] (.5,2) circle (.25ex);
    \draw[black,fill=black] (1,1.5) circle (.25ex);
    \draw[black,fill=black] (.75,2.5) circle (.25ex);

    \draw [thick] (0,0.5) -- (.5,1);
    \draw [thick] (0.5,1) -- (0,1.5);
    
    \draw [thick] (0,1.5) -- (0.5,2) -- (0.75,2.5);
    \draw [thick] (0.5,2) -- (0,2.5);
    \draw [thick] (0,3) -- (.75,2.5);

    \draw [thick] (.5,1) -- (1,1.5);
    \draw [thick] (1,1.5) -- (.5,2);

 \node at (-.25,0) {\(\sigma_2\)};
 \node at (-.25,.5) {\(\sigma_5\)};
 \node at (-.25,1) {\(\sigma_3\)};
 \node at (-.25,1.5) {\(\sigma_6\)};
 \node at (-.25,2) {\(\sigma_1\)};
 \node at (-.25,2.5) {\(\sigma_4\)};
 \node at (-.25,3) {\(\sigma_7\)};
 \node at (-.4,3.5) {\(\sigma_{10}\)};

 \node at (.5,.75) {\(\sigma_{8}\)};
 \node at (.5,1.75) {\(\sigma_{9}\)};
 \node at (1.2,1.25) {\(\sigma_{11}\)};
 \node at (1.1,2.5) {\(\sigma_{12}\)};
\end{tikzpicture}
\caption{The Hasse diagram of \(P_{12,\{1,4,7\}}^{513624}\).}
\label{fig1}
\end{figure}

For the explicit definition of \(P_{n,S}^\pi\) in general, let \(\eta=\pi^{-1}\) and take the transitive closure of the \(k\) chains of inequalities on the set \(\{\sigma_1,\ldots, \sigma_n\}\) obtained for each \(i\in S\):
\begin{equation}\label{eq:chain}\sigma_{i-1+\eta_1}<\sigma_{i-1+\eta_2}<\cdots<\sigma_{i-1+\eta_m}.\end{equation}

\subsection{Posets for non-overlapping permutations}
\label{ss:posetstructure}
The cluster posets of non-overlapping permutations have a particularly simple structure. First, note that if \(\pi\in\mfS_m\) is non-overlapping, then \(r_{n,k}^\pi=0\) unless \(n=1+k(m-1)\). This is because in order to have $k$ occurrences of \(\pi\) form a cluster, each one must overlap the next one on exactly one letter, and so each occurrence of \(\pi\) after the first adds \(m-1\) new letters.
Letting $$S(k,m):=\{1,m,2m-1,\ldots, 1+ (k-1)(m-1)\},$$ 
the same argument shows that
\(r_{n,S}^\pi=0\) unless \(S=S(k,m)\) for some $k$, and \(n=1+k(m-1)\). 
More generally, regardless of whether or not \(\pi\in\mfS_m\) is non-overlapping, the only set $S$ with $|S|=k$ for which \(r_{1+k(m-1),S}^\pi\neq0\) is \(S=S(k,m)\). In particular, $$r_{1+k(m-1),k}^\pi=r_{1+k(m-1),S(k,m)}^\pi.$$
Next we look more closely at the structure of the corresponding cluster poset \(P_{1+k(m-1),S(k,m)}^\pi\), which we will denote $P^\pi_k$ for short.

Suppose that \(\pi\in\mfS_m\) is in standard form and let \(a=\pi_1\) and \(b=\pi_m\). The poset \(P_k^\pi\) is generated by the $k$ chains of inequalities~\eqref{eq:chain}, where \(\eta=\pi^{-1}\), and $i=1+j(m-1)$ for $j=0,1,\dots,k-1$. Each one of these chains intersects the next one in one element. More precisely, the $b$th lowest element of the $j$th chain, which is $\sigma_{1+j(m-1)-1+\eta_b}=\sigma_{j(m-1)+m}$, coincides with the $a$th lowest element of the $(j+1)$st chain, which is $\sigma_{1+(j+1)(m-1)-1+\eta_a}=\sigma_{(j+1)(m-1)+1}$.

Arranging these $k$ chains with their identified elements, we can view the poset \(P_k^\pi\) as consisting of one long chain \(C\) with \(b+(k-2)(b-a)+m-a\) nodes, 
 together with \(k-1\) additional chains \(D_1,\ldots, D_{k-1}\) with \(m-b+a\) nodes. The chains $D_i$ are disjoint, and each of them intersects $C$ at one node, which is the \(a\)-th smallest element of \(D_i\) and the \((b+(i-1)(b-a))\)-th smallest element of \(D\). 
The Hasse diagrams of the posets $P^\pi_{3}$, $P^\pi_{4}$, $P^\pi_{5}$ for \(\pi=34671285\) are shown in Figure \ref{fig2}. A more general drawing of the Hasse diagram of \(P_{4}^\pi\) for arbitrary \(\pi\in\mfS_m\) in standard form is given in Figure \ref{fig3}, where the \(k-1\) chains \(D_1,\ldots,D_{k-1}\) are drawn diagonally and the chain \(C\) is drawn vertically.
We state an immediate consequence of  the above description, which will be used in the proof of Theorem~\ref{thm1}. 
\begin{lem}\label{lem:only1m}
For non-overlapping $\pi\in\mfS_m$, the poset \(P_{k}^\pi\) depends only on $\pi_1$ and $\pi_m$, up to isomorphism.
\end{lem}
This description of the cluster posets \(P_{k}^\pi\) will also be useful in Section~\ref{sec:proof2} when proving Lemma~\ref{lemma} and Theorem~\ref{thm2}.

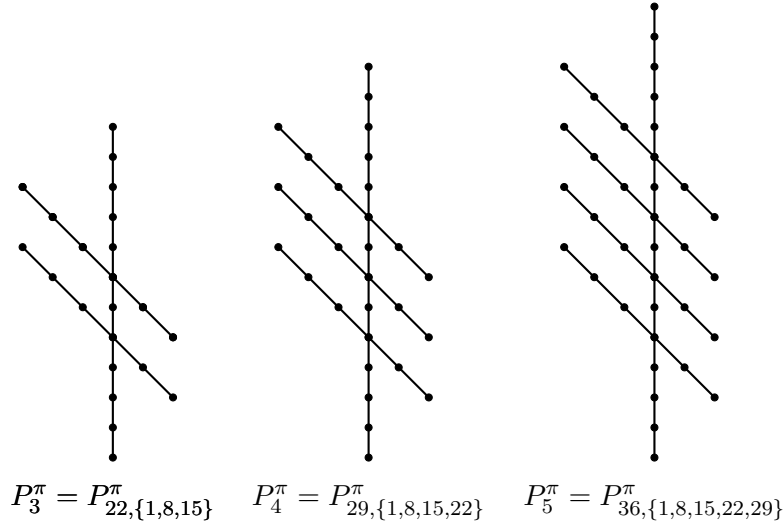
\begin{figure}[htb]
\centering
%k=3
\begin{tikzpicture}[scale=.4]
    \draw [thick] (0,-1) -- (0,10);
    \draw[black,fill=black] (0,-1) circle (.7ex);
    \draw[black,fill=black] (0,0) circle (.7ex);
    \draw[black,fill=black] (0,1) circle (.7ex);
    \draw[black,fill=black] (0,2) circle (.7ex);
    \draw[black,fill=black] (0,3) circle (.7ex);
    \draw[black,fill=black] (0,4) circle (.7ex);
    \draw[black,fill=black] (0,5) circle (.7ex);
    \draw[black,fill=black] (0,6) circle (.7ex);
    \draw[black,fill=black] (0,7) circle (.7ex);
    \draw[black,fill=black] (0,8) circle (.7ex);
    \draw[black,fill=black] (0,9) circle (.7ex);
    \draw[black,fill=black] (0,10) circle (.7ex);

    \draw [thick] (-3,6) -- (2,1);
    \draw[black,fill=black] (-3,6) circle (.7ex);
    \draw[black,fill=black] (-2,5) circle (.7ex);
    \draw[black,fill=black] (-1,4) circle (.7ex);
    \draw[black,fill=black] (0,3) circle (.7ex);
    \draw[black,fill=black] (1,2) circle (.7ex);
    \draw[black,fill=black] (2,1) circle (.7ex);

    \draw [thick] (-3,8) -- (2,3);
    \draw[black,fill=black] (-3,8) circle (.7ex);
    \draw[black,fill=black] (-2,7) circle (.7ex);
    \draw[black,fill=black] (-1,6) circle (.7ex);
    \draw[black,fill=black] (0,5) circle (.7ex);
    \draw[black,fill=black] (1,4) circle (.7ex);
    \draw[black,fill=black] (2,3) circle (.7ex);
    \node at (0,-2.5) {\(P^\pi_3=P^{\pi}_{22,\{1,8,15\}}\)};

    \draw[black,fill=black] (-3,8) circle (.7ex);
    \draw[black,fill=black] (-2,7) circle (.7ex);
    \draw[black,fill=black] (-1,6) circle (.7ex);
    \draw[black,fill=black] (0,5) circle (.7ex);
    \draw[black,fill=black] (1,4) circle (.7ex);
    \draw[black,fill=black] (2,3) circle (.7ex);
    \node at (0,-2.5) {\(P^\pi_3=P^{\pi}_{22,\{1,8,15\}}\)};

\begin{scope}[shift={(2.5,0)}]
%k=4
  \draw [thick] (6,-1) -- (6,12);
    \draw[black,fill=black] (6,-1) circle (.7ex);
    \draw[black,fill=black] (6,0) circle (.7ex);
    \draw[black,fill=black] (6,1) circle (.7ex);
    \draw[black,fill=black] (6,2) circle (.7ex);
    \draw[black,fill=black] (6,3) circle (.7ex);
    \draw[black,fill=black] (6,4) circle (.7ex);
    \draw[black,fill=black] (6,5) circle (.7ex);
    \draw[black,fill=black] (6,6) circle (.7ex);
    \draw[black,fill=black] (6,7) circle (.7ex);
    \draw[black,fill=black] (6,8) circle (.7ex);
    \draw[black,fill=black] (6,9) circle (.7ex);
    \draw[black,fill=black] (6,10) circle (.7ex);
    \draw[black,fill=black] (6,11) circle (.7ex);
    \draw[black,fill=black] (6,12) circle (.7ex);

    \draw [thick] (3,6) -- (8,1);
    \draw[black,fill=black] (3,6) circle (.7ex);
    \draw[black,fill=black] (4,5) circle (.7ex);
    \draw[black,fill=black] (5,4) circle (.7ex);
    \draw[black,fill=black] (6,3) circle (.7ex);
    \draw[black,fill=black] (7,2) circle (.7ex);
    \draw[black,fill=black] (8,1) circle (.7ex);

    \draw [thick] (3,8) -- (8,3);
    \draw[black,fill=black] (3,8) circle (.7ex);
    \draw[black,fill=black] (4,7) circle (.7ex);
    \draw[black,fill=black] (5,6) circle (.7ex);
    \draw[black,fill=black] (6,5) circle (.7ex);
    \draw[black,fill=black] (7,4) circle (.7ex);
    \draw[black,fill=black] (8,3) circle (.7ex);

    \draw [thick] (3,10) -- (8,5);
    \draw[black,fill=black] (3,10) circle (.7ex);
    \draw[black,fill=black] (4,9) circle (.7ex);
    \draw[black,fill=black] (5,8) circle (.7ex);
    \draw[black,fill=black] (6,7) circle (.7ex);
    \draw[black,fill=black] (7,6) circle (.7ex);
    \draw[black,fill=black] (8,5) circle (.7ex);
    \node at (6,-2.5) {\(P^\pi_4=P^{\pi}_{29,\{1,8,15,22\}}\)};
\end{scope}

\begin{scope}[shift={(5,0)}]
%k=5
  \draw [thick] (13,-1) -- (13,14);
    \draw[black,fill=black] (13,-1) circle (.7ex);
    \draw[black,fill=black] (13,0) circle (.7ex);
    \draw[black,fill=black] (13,1) circle (.7ex);
    \draw[black,fill=black] (13,2) circle (.7ex);
    \draw[black,fill=black] (13,3) circle (.7ex);
    \draw[black,fill=black] (13,4) circle (.7ex);
    \draw[black,fill=black] (13,5) circle (.7ex);
    \draw[black,fill=black] (13,6) circle (.7ex);
    \draw[black,fill=black] (13,7) circle (.7ex);
    \draw[black,fill=black] (13,8) circle (.7ex);
    \draw[black,fill=black] (13,9) circle (.7ex);
    \draw[black,fill=black] (13,10) circle (.7ex);
    \draw[black,fill=black] (13,11) circle (.7ex);
    \draw[black,fill=black] (13,12) circle (.7ex);
    \draw[black,fill=black] (13,13) circle (.7ex);
    \draw[black,fill=black] (13,14) circle (.7ex);

    \draw [thick] (10,6) -- (15,1);
    \draw[black,fill=black] (10,6) circle (.7ex);
    \draw[black,fill=black] (11,5) circle (.7ex);
    \draw[black,fill=black] (12,4) circle (.7ex);
    \draw[black,fill=black] (13,3) circle (.7ex);
    \draw[black,fill=black] (14,2) circle (.7ex);
    \draw[black,fill=black] (15,1) circle (.7ex);

    \draw [thick] (10,8) -- (15,3);
    \draw[black,fill=black] (10,8) circle (.7ex);
    \draw[black,fill=black] (11,7) circle (.7ex);
    \draw[black,fill=black] (12,6) circle (.7ex);
    \draw[black,fill=black] (13,5) circle (.7ex);
    \draw[black,fill=black] (14,4) circle (.7ex);
    \draw[black,fill=black] (15,3) circle (.7ex);

    \draw [thick] (10,10) -- (15,5);
    \draw[black,fill=black] (10,10) circle (.7ex);
    \draw[black,fill=black] (11,9) circle (.7ex);
    \draw[black,fill=black] (12,8) circle (.7ex);
    \draw[black,fill=black] (13,7) circle (.7ex);
    \draw[black,fill=black] (14,6) circle (.7ex);
    \draw[black,fill=black] (15,5) circle (.7ex);

    \draw [thick] (10,12) -- (15,7);
    \draw[black,fill=black] (10,12) circle (.7ex);
    \draw[black,fill=black] (11,11) circle (.7ex);
    \draw[black,fill=black] (12,10) circle (.7ex);
    \draw[black,fill=black] (13,9) circle (.7ex);
    \draw[black,fill=black] (14,8) circle (.7ex);
    \draw[black,fill=black] (15,7) circle (.7ex);
    \node at (13,-2.5) {\(P^\pi_5=P^{\pi}_{36,\{1,8,15,22,29\}}\)};
\end{scope}
\end{tikzpicture}
\caption{Some cluster posets for the non-overlapping permutation $\pi=34671285$.}
\label{fig2}
\end{figure}

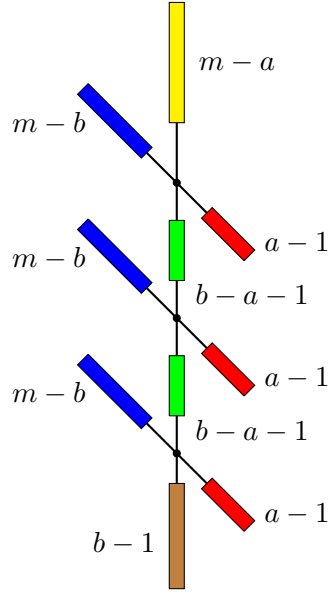
\begin{figure}[htb]
\centering
%k=3
\begin{tikzpicture}[scale=.4]
    \draw [thick] (0,2) -- (0,4.25);
    \draw [thick] (0,6.25)--(0,8.75);
    \draw [thick] (0,10.75) -- (0,14);
    \draw[black,fill=black] (0,3) circle (.7ex);
    \draw[black,fill=black] (0,7.5) circle (.7ex);
    \draw[black,fill=black] (0,12) circle (.7ex);

    \draw [thick] (-1,4) -- (1,2);
    \draw [thick] (-1,8.5) -- (1,6.5);
    \draw [thick] (-1,13) -- (1,11);
%rotate around={45:(4,-1)}
    %lower arms
\draw[rotate around={45:(1,2)},shift={(.75,0)}, fill=red ] (0,0) -- (.5,0) -- (.5,2) -- (0,2) -- (0,0);
\draw[rotate around={45:(1,6.5)},shift={(.75,4.5)}, fill=red] (0,0) -- (.5,0) -- (.5,2) -- (0,2) -- (0,0);
\draw[rotate around={45:(1,11)},shift={(.75,9)}, fill=red ] (0,0) -- (.5,0) -- (.5,2) -- (0,2) -- (0,0);
\node at (4,1) {\(a-1\)};
\node at (4,5.5) {\(a-1\)};
\node at (4,10) {\(a-1\)};

    %upper arms
\draw[rotate around={225:(-1,4)},shift={(-1.25,1)}, fill=blue] (0,0) -- (.5,0) -- (.5,3) -- (0,3) -- (0,0);
\draw[rotate around={225:(-1,8.5)},shift={(-1.25,5.5)}, fill=blue] (0,0) -- (.5,0) -- (.5,3) -- (0,3) -- (0,0);
\draw[rotate around={225:(-1,13)},shift={(-1.25,10)}, fill=blue ] (0,0) -- (.5,0) -- (.5,3) -- (0,3) -- (0,0);
\node at (-4.25,5) {\(m-b\)};
\node at (-4.25,9.5) {\(m-b\)};
\node at (-4.25,14) {\(m-b\)};

%middle chain
\draw[shift={(-.25,4.25)}, fill=green] (0,0) -- (.5,0) -- (.5,2) -- (0,2) -- (0,0);
\draw[shift={(-.25,8.75)}, fill=green] (0,0) -- (.5,0) -- (.5,2) -- (0,2) -- (0,0);
\node at (2.5,3.75) {\(b-a-1\)};
\node at (2.5,8.25) {\(b-a-1\)};

%top chain
\draw[shift={(-.25,14)}, fill=yellow] (0,0) -- (.5,0) -- (.5,4) -- (0,4) -- (0,0);
\node at (2,16) {\(m-a\)};

%bottom chain
\draw[shift={(-.25,-1.5)}, fill=brown] (0,0) -- (.5,0) -- (.5,3.5) -- (0,3.5) -- (0,0);
\node at (-1.75,0) {\(b-1\)};

\end{tikzpicture}
\caption{The cluster poset \(P_{4}^\pi\) of a permutation \(\pi\in\mfS_m\) in standard form with \(\pi_1=a\) and \(\pi_m=b\).}
\label{fig3}
\end{figure}

\section{A more general conjecture}
\label{sec:proof1}

In this section we prove Theorem \ref{thm1}, which states that if Conjecture~\ref{elizaldeconj} is true, then so is
Conjecture~\ref{dwyerconj}. This conjecture hypothesizes that any two strongly c-Wilf equivalent patterns in standard form must have the same first and last letter.

\begin{proof}[Proof of Theorem \ref{thm1}]
It is easy to check that Conjecture \ref{dwyerconj} holds for \(m\le 4\). Indeed, there are only two permutations in $\mfS_3$ standard form, namely $123$ and $132$, and they are not c-Wilf equivalent. There are \(8\) permutations in \(\mfS_4\) in standard form, namely $1234$, $1243$, $1324$, $1342$, $1423$, $1432$, $2143$, and $2413$, and, as shown in~\cite{EN2003},  the only two that are c-Wilf equivalent are \(1342\) and \(1432\), which have the same first and last letter. 

Now suppose that \(m\ge5\) and \(\pi,\tau\in\mfS_m\) are in standard form and $\pi\swe\tau$. By Theorem~\ref{clustermethod}, \(r_{n,k}^\pi=r_{n,k}^\tau\) for all \(n\) and \(k\). In particular, taking \(n=1+k(m-1)\), we have \(r_{1+k(m-1),k}^\pi=r_{1+k(m-1),k}^\tau\) for all \(k\). 

If \(\pi_m=m\), then the condition  \(\pi_1+\pi_m\le m+1\) forces \(\pi_1=1\).  In this case, \(P_{k}^\pi\) is a chain, and
\(r_{1+k(m-1),k}^\pi=1=r_{1+k(m-1),k}^\tau\) for all \(k\ge1\). This can only happen if \(P_{k}^\tau\) is a chain as well, which forces \(\tau_1=1\) and \(\tau_m=m\). A symmetric argument shows that if $\tau_m=m$, then $\pi_1=\tau_1=1$ and $\pi_m=\tau_m=m$. 
 
We are left with the case \(\pi_m, \tau_m<m\). 
 In this case, we construct two non-overlapping permutations \(p,t\in\mfS_m\) with the same first and last letters as \(\pi\) and \(\tau\), respectively, following a construction from~\cite{ElizaldeMostLeast2013}:
\begin{align*}
p&=\pi_1(\pi_1+1)\dots(\pi_m-1)(\pi_m+1)\dots(m-1)12\dots(\pi_1-1)m\pi_m,\\
t&=\tau_1(\tau_1+1)\dots(\tau_m-1)(\tau_m+1)\dots(m-1)12\dots(\tau_1-1)m\tau_m.
\end{align*}
By Lemma~\ref{lem:only1m}, the poset \(P_{k}^\pi\) depends only on $\pi_1$ and $\pi_m$. It follows that  \(r_{1+k(m-1),k}^\pi=r_{1+k(m-1),k}^p\) and \(r_{1+k(m-1),k}^\tau=r_{1+k(m-1),k}^t\) for all \(k\). Since \(p\) and \(t\) are non-overlapping, these are their only non-zero cluster numbers, and so $p\we t$ by Theorem~\ref{clustermethod}. Now Conjecture~\ref{elizaldeconj} states that \(p\) and \(t\) must have the same first and last letter, and thus the same holds for \(\pi\) and \(\tau\), implying Conjecture~\ref{dwyerconj}. \end{proof}

\section{Asymptotic growth of non-overlapping cluster numbers}
\label{sec:proof2}

Our goal in this section is to prove the following strengthening of Theorem~\ref{thm2}. 

\begin{thm}
\label{thm2strong}
Let \(\pi,\tau\in\mfS_m\) be in standard form. If \(\pi_m-\pi_1>\tau_m-\tau_1\), then there is an integer \(K\) such that 
$$r_{1+k(m-1),k}^\pi<r_{1+k(m-1),k}^\tau$$ for all \(k\ge K\). 
In particular, $\pi\not\swe\tau$.
\end{thm}

Theorem~\ref{thm2} clearly follows from Theorem~\ref{thm2strong}. To prove Theorem~\ref{thm2strong}, we will analyze the asymptotic behavior of \(r_{1+k(m-1),k}^\pi\), which will allow us to extract information about the quantity \(\pi_m-\pi_1\) from this sequence.

\begin{lem}
\label{lemma}
Let \(\pi\in\mfS_m\) be in standard form. Then, as $k$ tends to infinity,
$$\left(r_{1+k(m-1),k}^\pi\right)^{1/k}=O(k^{m-\pi_m+\pi_1-1}).$$
\end{lem}

\begin{proof}
We will show that there exist positive constants \(L_\pi,U_\pi\) and $K$ such that for all $k\ge K$,
\begin{equation}\label{eq:LU} 
L_\pi\, k^{m-\pi_m+\pi_1-1}\le \left(r_{1+k(m-1),k}^\pi\right)^{1/k}\le U_\pi\, k^{m-\pi_m+\pi_1-1}. 
\end{equation}

First we note that if we have two posets \(R\) and \(Q\) on the same set \(X\) with order relations \(\le_R\) and \(\le_Q\) such that $x\le_R y$ implies $x\le_Q y$ for all \(x,y\in X\), then \(R\) has at least as many linear extensions as \(Q\). We obtain upper and lower bounds for \(r_{1+k(m-1),k}^\pi\) by removing and adding relations to \(P_{k}^\pi\), which has \(r_{1+k(m-1),k}^\pi\) linear extensions by construction, and counting the number of linear extensions of the resulting modified posets. We use the notation introduced in Section \ref{ss:posetstructure} throughout the proof, including $a=\pi_1$ and $b=\pi_m$.
 It is helpful to refer to Figure~\ref{fig3} and to think of the constructions of the posets below as adding (or removing) relations between the rectangles in this figure.

We will build two new posets $U_k^\pi$ and $L_k^\pi$ with $\ell_k^\pi$ and $u_k^\pi$ linear extensions, respectively, such that \(\ell_k^\pi\le r_{1+k(m-1),k}^\pi\le u_{k}^\pi\). Then we will show that, as \(k\to\infty\), 
$$\left(\ell_k^\pi\right)^{1/k}\sim N k^{m-b+a-1} \mbox{ and }\left(u_k^\pi\right)^{1/k}\sim Mk^{m-b+a-1}$$  
for some positive constants \(N\) and \(M\), where in this proof we use the notation $a_k\sim b_k$ to mean $\lim_{k\to\infty}a_k/b_k=1$. It will follow that the sequence $\left(r_{1+k(m-1),k}^\pi\right)^{1/k}/k^{m-b+a-1}$ is bounded away from \(0\) and \(\infty\) as \(k\to\infty\), which is equivalent to the existence of \(L_\pi\), \(U_\pi\) and $K$.

\medskip

\noindent {\sl Upper bound:} For each \(i\in[k-1]\), let \(T_i\) be the \(a-1\) smallest elements of \(D_i\), corresponding to the red rectangles in Figure~\ref{fig3}. We remove all relations between elements from \(T_i\) and elements from \(P_{k}^\pi\setminus T_i\) for all $i$, to form a new poset \(U_{k}^\pi\) with at least \(r_{1+k(m-1),k}^\pi\) linear extensions. As an example, the Hasse diagram of \(U^{\pi}_{4}\) is given on the left of Figure~\ref{fig4}.

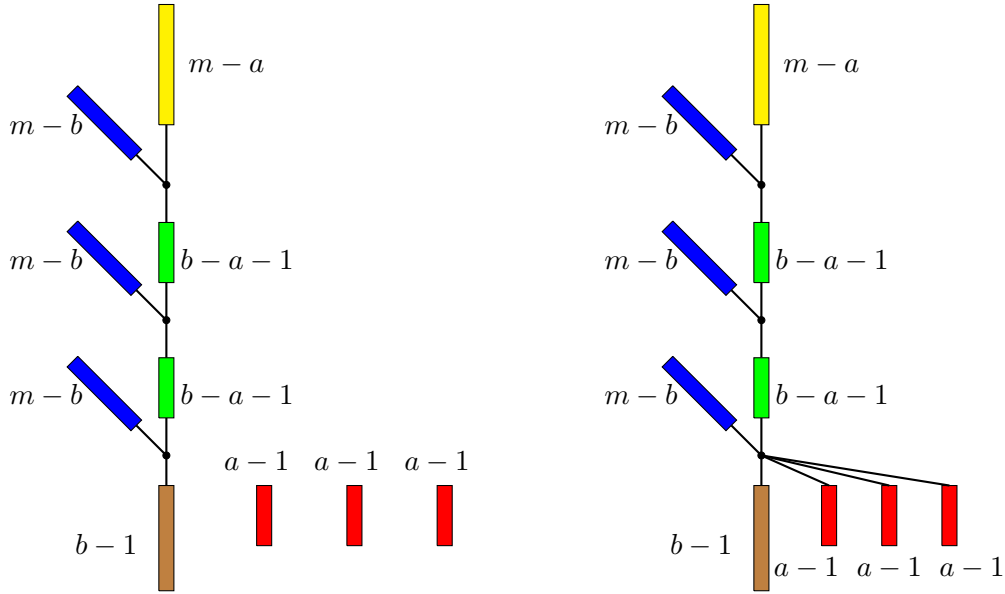
\begin{figure}[htb]
\centering
%k=3
\begin{tikzpicture}[scale=.4]
    \draw [thick] (0,2) -- (0,4.25);
    \draw [thick] (0,6.25)--(0,8.75);
    \draw [thick] (0,10.75) -- (0,14);
    \draw[black,fill=black] (0,3) circle (.7ex);
    \draw[black,fill=black] (0,7.5) circle (.7ex);
    \draw[black,fill=black] (0,12) circle (.7ex);

    \draw [thick] (-1,4) -- (0,3);
    \draw [thick] (-1,8.5) -- (0,7.5);
    \draw [thick] (-1,13) -- (0,12);
%rotate around={45:(4,-1)}
    %lower arms
\draw[fill=red,shift={(3,0)} ] (0,0) -- (.5,0) -- (.5,2) -- (0,2) -- (0,0);
\draw[fill=red,shift={(6,0)}] (0,0) -- (.5,0) -- (.5,2) -- (0,2) -- (0,0);
\draw[fill=red,shift={(9,0)} ] (0,0) -- (.5,0) -- (.5,2) -- (0,2) -- (0,0);
\node at (3,2.8) {\(a-1\)};
\node at (6,2.8) {\(a-1\)};
\node at (9,2.8) {\(a-1\)};

    %upper arms
\draw[fill=blue,rotate around={225:(-1,4)},shift={(-1.25,1)}, ] (0,0) -- (.5,0) -- (.5,3) -- (0,3) -- (0,0);
\draw[fill=blue,rotate around={225:(-1,8.5)},shift={(-1.25,5.5)}, ] (0,0) -- (.5,0) -- (.5,3) -- (0,3) -- (0,0);
\draw[fill=blue,rotate around={225:(-1,13)},shift={(-1.25,10)}, ] (0,0) -- (.5,0) -- (.5,3) -- (0,3) -- (0,0);
\node at (-4,5) {\(m-b\)};
\node at (-4,9.5) {\(m-b\)};
\node at (-4,14) {\(m-b\)};

%middle chain
\draw[fill=green,shift={(-.25,4.25)}, ] (0,0) -- (.5,0) -- (.5,2) -- (0,2) -- (0,0);
\draw[fill=green,shift={(-.25,8.75)}, ] (0,0) -- (.5,0) -- (.5,2) -- (0,2) -- (0,0);
\node at (2.35,5) {\(b-a-1\)};
\node at (2.35,9.5) {\(b-a-1\)};

%top chain
\draw[fill=yellow,shift={(-.25,14)}, ] (0,0) -- (.5,0) -- (.5,4) -- (0,4) -- (0,0);
\node at (2,16) {\(m-a\)};

%bottom chain
\draw[fill=brown,shift={(-.25,-1.5)}, ] (0,0) -- (.5,0) -- (.5,3.5) -- (0,3.5) -- (0,0);
\node at (-2,0) {\(b-1\)};
\end{tikzpicture}\hspace{40pt}
\begin{tikzpicture}[scale=.4]
    \draw [thick] (0,3) -- (0,4.25);
    \draw [thick] (0,6.25)--(0,8.75);
    \draw [thick] (0,10.75) -- (0,14);
    \draw[black,fill=black] (0,3) circle (.7ex);
    \draw[black,fill=black] (0,7.5) circle (.7ex);
    \draw[black,fill=black] (0,12) circle (.7ex);

    \draw [thick] (-1,4) -- (0,3);
    \draw [thick] (-1,8.5) -- (0,7.5);
    \draw [thick] (-1,13) -- (0,12);
%rotate around={45:(4,-1)}
    %lower arms
\draw[fill=red,shift={(2,0)} ] (0,0) -- (.5,0) -- (.5,2) -- (0,2) -- (0,0);
\draw[fill=red,shift={(4,0)}] (0,0) -- (.5,0) -- (.5,2) -- (0,2) -- (0,0);
\draw[fill=red,shift={(6,0)} ] (0,0) -- (.5,0) -- (.5,2) -- (0,2) -- (0,0);
\node at (1.5,-0.75) {\(a-1\)};
\node at (4.25,-0.75) {\(a-1\)};
\node at (7,-0.75) {\(a-1\)};

    %upper arms
\draw[fill=blue,rotate around={225:(-1,4)},shift={(-1.25,1)}, ] (0,0) -- (.5,0) -- (.5,3) -- (0,3) -- (0,0);
\draw[fill=blue,rotate around={225:(-1,8.5)},shift={(-1.25,5.5)}, ] (0,0) -- (.5,0) -- (.5,3) -- (0,3) -- (0,0);
\draw[fill=blue,rotate around={225:(-1,13)},shift={(-1.25,10)}, ] (0,0) -- (.5,0) -- (.5,3) -- (0,3) -- (0,0);
\node at (-4,5) {\(m-b\)};
\node at (-4,9.5) {\(m-b\)};
\node at (-4,14) {\(m-b\)};

%middle chain
\draw[fill=green,shift={(-.25,4.25)}, ] (0,0) -- (.5,0) -- (.5,2) -- (0,2) -- (0,0);
\draw[fill=green,shift={(-.25,8.75)}, ] (0,0) -- (.5,0) -- (.5,2) -- (0,2) -- (0,0);
\node at (2.35,5) {\(b-a-1\)};
\node at (2.35,9.5) {\(b-a-1\)};

%top chain
\draw[fill=yellow,shift={(-.25,14)}, ] (0,0) -- (.5,0) -- (.5,4) -- (0,4) -- (0,0);
\node at (2,16) {\(m-a\)};

%bottom chain
\draw[fill=brown,shift={(-.25,-1.5)}, ] (0,0) -- (.5,0) -- (.5,3.5) -- (0,3.5) -- (0,0);
\node at (-2,0) {\(b-1\)};

%connectives at the bottom
\draw [thick] (0,3) -- (0,2);
\draw [thick] (0,3) -- (2.25,2);
\draw [thick] (0,3) -- (4.25,2);
\draw [thick] (0,3) -- (6.25,2);

\end{tikzpicture}
\caption{The posets \(U_{4}^\pi\) (left) and \(L_4^\pi\) (right) of a permutation \(\pi\in\mfS_m\) in standard form with \(\pi_1=a\) and \(\pi_m=b\).}
\label{fig4}
\end{figure}

The number of linear extensions of \(U_k^\pi\) is $u_k^\pi=s_k^\pi q_k^\pi$, where
$$s_k^\pi=\binom{1+k(m-1)}{a-1,\dots a-1, 1+k(m-1)-(k-1)(a-1)}=\frac{(1+k(m-1))!}{(a-1)!^{k-1}(a+k(m-a))!}$$
and
$$q_k^\pi=\prod_{i=1}^{k-1}\binom{i(m-a)+m-b}{m-b}=\frac{1}{(m-b)!^{k-1}}\prod_{i=1}^{k-1}(i(m-a)+1)\cdots(i(m-a)+m-b).$$

As \(k\to\infty\), Stirling's formula gives 
\begin{equation}\label{eq:s}\left(s_k^\pi\right)^{1/k}\sim\frac{(m-1)^{m-1}}{e^{a-1}(a-1)!(m-a)^{m-a}}\, k^{a-1}.\end{equation}

To bound $q_k^\pi$, we use the fact that $a<b$ to obtain
$$(i(m-a))^{m-b}\le (i(m-a)+1)\cdots(i(m-a)+m-b) \le ((i+1)(m-a))^{m-b},$$
and so
\[
\frac{(k-1)!^{m-b}(m-a)^{(m-b)(k-1)}}{(m-b)!^{k-1}}\, \le\, q_k^\pi \, \le\, \frac{k!^{m-b}(m-a)^{(m-b)(k-1)}}{(m-b)!^{k-1}}.
\]
Applying Stirling's formula again as
\(k\to\infty\), we obtain 
\begin{equation}\label{eq:q}\left(q_k^\pi\right)^{1/k}\sim\frac{(m-a)^{m-b}}{e^{m-b}(m-b)!}\, k^{m-b}.\end{equation}
Combining~\eqref{eq:s} and~\eqref{eq:q}, we get
$$\left(u_k^\pi\right)^{1/k}\sim\frac{(m-1)^{m-1}}{e^{m-b+a-1}(a-1)!(m-b)!(m-a)^{b-a}}\, k^{m-b+a-1}$$ as desired.

\medskip

\noindent {\sl Lower bound:} Again, we modify the relations between elements of \(T_i\) and the rest of the poset $P_k^\pi$. This time we add relations to force every element in each \(T_i\) to be smaller than the \(b\)-th smallest element in \(C\). Let \(L_{k}^\pi\) be the resulting poset. As an example, the Hasse diagram of \(L^{\pi}_{4}\) is given on the right of Figure~\ref{fig4}.

The number of linear extensions \(L_k^\pi\) is $\ell_k^\pi=t_k^\pi q_k^\pi$, where $q_k^\pi$ is the same as before, and
$$t_k^\pi=\binom{b-1+(k-1)(a-1)}{b-1,a-1,a-1,\ldots, a-1}.$$
Again using Stirling's formula we see that, as \(k\to\infty\),
$$\left(t_k^\pi\right)^{1/k}\sim \frac{(a-1)^{a-1}}{e^{a-1}(a-1)!}\,k^{a-1},$$
and so
\[
\left(\ell_k^\pi\right)^{1/k}\sim \frac{(a-1)^{a-1} (m-a)^{m-b}}{e^{m-b+a-1}(a-1)!(m-b)!} \, k^{m-b+a-1}.\qedhere
\]
\end{proof}

\begin{proof}[Proof of Theorem \ref{thm2strong}]
Let \(\pi,\tau\in\mfS_m\) be in standard form, and suppose that \(\pi_m-\pi_1>\tau_m-\tau_1\). 
By Lemma~\ref{lemma}, there exist constants $U_\pi>0$ and $K_1$ such that
$$\left(r_{1+k(m-1),k}^\pi\right)^{1/k}\le U_\pi\, k^{m-\pi_m+\pi_1-1}$$ for $k\ge K_1$, and constants
$L_\tau>0$ and $K_2$ such that 
$$L_\tau\, k^{m-\tau_m+\tau_1-1}\le \left(r_{1+k(m-1),k}^\tau\right)^{1/k}$$ for $k\ge K_2$.
Since \(m-\pi_m+\pi_1-1<m-\tau_m+\tau_1-1\), there exists \(K_3\) such that
\(
U_\pi k^{m-\pi_m+\pi_1-1}<L_\tau k^{m-\tau_m+\tau_1-1}
\)
for all \(k\ge K_3\). Then, for $k\ge \max\{K_1,K_2,K_3\}$,
$$r_{1+k(m-1),k}^\pi\le\left(U_\pi k^{m-\pi_m+\pi_1-1}\right)^k<\left(L_\tau k^{m-\tau_m+\tau_1-1}\right)^k \le r_{1+k(m-1),k}^\tau.$$
In particular, by Theorem~\ref{clustermethod}, $\pi$ and $\tau$ cannot be strongly c-Wilf equivalent.\end{proof}

It follows from Lemma~\ref{lemma} that, for $\pi\in\mfS_m$, the difference $\pi_m-\pi_1$ can be recovered from the sequence of cluster numbers $r_{1+k(m-1),k}^\pi$ by the formula
$$\pi_m-\pi_1=m-1-\lim_{k\to\infty}\frac{\log r_{1+k(m-1),k}^\pi}{k \log k},$$
and that the limit is guaranteed to exist. This can easily be seen by taking logarithms in Equation~\eqref{eq:LU}, dividing by $\log k$, and making $k$ tend to infinity.

\begin{ex}
Consider the permutations in standard form $\pi=23567184$, $\tau=34671285$ and $\chi=35671284$.
By Theorem~\ref{thm2strong}, 
$$r_{1+k(m-1),k}^\chi<r_{1+k(m-1),k}^\pi \mbox{\quad and \quad} r_{1+k(m-1),k}^\chi<r_{1+k(m-1),k}^\tau$$
for $k$ large enough.
Unfortunately, Theorem~\ref{thm2strong} tells us nothing about the relationship between \(r_{1+k(m-1),k}^\pi\) and \(r_{1+k(m-1),k}^\tau\), since \(\pi_8-\pi_1=2=\tau_8-\tau_1\). 
\end{ex}

The comparison between cluster numbers $r_{1+k(m-1),k}^\pi$ for different permutations with a fixed difference $\pi_m-\pi_1$ is not given 
by Theorem~\ref{thm2strong}, and it is open in general. For the special case of \(\pi_m-\pi_1=1\) and $k=2$, the following relationship was proved in~\cite{ElizaldeMostLeast2013}. 

\begin{prop}[\cite{ElizaldeMostLeast2013}]
\label{prop:diff1}
Let \(\pi,\tau\in\mfS_m\) be in standard form, and suppose that \(\pi_m-\pi_1=\tau_m-\tau_1=1\) and \(\pi_1<\tau_1\). Then  
$$r_{2m-1,2}^\pi>r_{2m-1,2}^\tau.$$ 
In particular, \(\pi\not\swe\tau\). 
\end{prop}

Recall that Conjecture~\ref{dwyerconj} states that the strong c-Wilf equivalence class of a permutation $\pi\in\mfS_m$ in standard form uniquely determines the values $\pi_1$ and $\pi_m$. Theorem~\ref{thm2strong} shows that this equivalence class determines the difference $\pi_m-\pi_1$, and Proposition~\ref{prop:diff1} shows that, in the case that this difference is $1$, it also determines $\pi_1$ (and thus $\pi_m$). For the remaining cases, we have the following conjecture. If true, it would settle Conjecture~\ref{dwyerconj}.

\begin{conj}
\label{conj:clusterbehavior}
Let \(\pi,\tau\in\mfS_m\) be in standard form, and suppose that  \(\pi_m-\pi_1=\tau_m-\tau_1\ge2\) and \(\pi_1<\tau_1\). Then there exists some $k$ such that $$r_{1+k(m-1),k}^\pi<r_{1+k(m-1),k}^\tau.$$
\end{conj}

It is observed in~\cite{ElizaldeMostLeast2013} that the cluster numbers of $\pi= 23567184$ and $\tau= 34671285$ coincide for 
$k=2$ but not for $k=3$. More precisely, $r_{15,2}^\pi=840=r_{15,2}^\tau$, as can be seen by counting the number of linear extensions of the posets in Figure~\ref{fig:equal_r2}, but $r_{15,3}^\pi=11642400<12153960=r_{15,3}^\tau$.

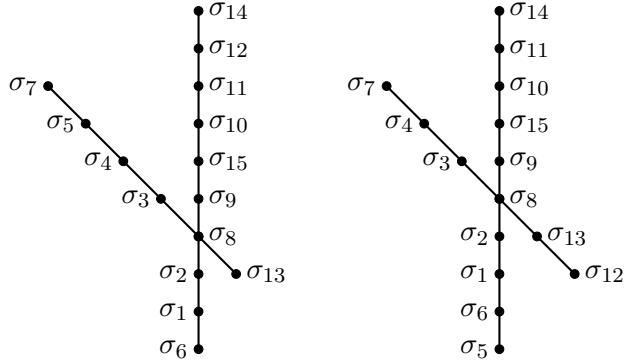
\begin{figure}
\centering
\begin{tikzpicture}[scale=.5]
  \draw [thick] (0,-1) -- (0,8);
    \draw[black,fill=black] (0,-1) circle (.7ex)  node[left] {$\sigma_6$};
    \draw[black,fill=black] (0,0) circle (.7ex) node[left] {$\sigma_1$};
    \draw[black,fill=black] (0,1) circle (.7ex) node[left] {$\sigma_2$};
    \draw[black,fill=black] (0,2) circle (.7ex) node[right] {$\sigma_8$};
    \draw[black,fill=black] (0,3) circle (.7ex) node[right] {$\sigma_9$};
    \draw[black,fill=black] (0,4) circle (.7ex) node[right] {$\sigma_{15}$};
    \draw[black,fill=black] (0,5) circle (.7ex) node[right] {$\sigma_{10}$};
    \draw[black,fill=black] (0,6) circle (.7ex) node[right] {$\sigma_{11}$};
    \draw[black,fill=black] (0,7) circle (.7ex) node[right] {$\sigma_{12}$};
    \draw[black,fill=black] (0,8) circle (.7ex) node[right] {$\sigma_{14}$};

    \draw [thick] (-4,6) -- (1,1);
    \draw[black,fill=black] (-4,6) circle (.7ex) node[left] {$\sigma_7$};
    \draw[black,fill=black] (-3,5) circle (.7ex) node[left] {$\sigma_5$};
    \draw[black,fill=black] (-2,4) circle (.7ex) node[left] {$\sigma_4$};
    \draw[black,fill=black] (-1,3) circle (.7ex) node[left] {$\sigma_{3}$};
    \draw[black,fill=black] (0,2) circle (.7ex);
    \draw[black,fill=black] (1,1) circle (.7ex) node[right] {$\sigma_{13}$};

\begin{scope}[shift={(8,0)}]
    \draw [thick] (0,-1) -- (0,8);
    \draw[black,fill=black] (0,-1) circle (.7ex)  node[left] {$\sigma_5$};
    \draw[black,fill=black] (0,0) circle (.7ex) node[left] {$\sigma_6$};
    \draw[black,fill=black] (0,1) circle (.7ex) node[left] {$\sigma_1$};
    \draw[black,fill=black] (0,2) circle (.7ex) node[left] {$\sigma_2$};
    \draw[black,fill=black] (0,3) circle (.7ex) node[right] {$\sigma_8$};
    \draw[black,fill=black] (0,4) circle (.7ex) node[right] {$\sigma_9$};
    \draw[black,fill=black] (0,5) circle (.7ex) node[right] {$\sigma_{15}$};
    \draw[black,fill=black] (0,6) circle (.7ex) node[right] {$\sigma_{10}$};
    \draw[black,fill=black] (0,7) circle (.7ex) node[right] {$\sigma_{11}$};
    \draw[black,fill=black] (0,8) circle (.7ex) node[right] {$\sigma_{14}$};

    \draw [thick] (-3,6) -- (2,1);
    \draw[black,fill=black] (-3,6) circle (.7ex) node[left] {$\sigma_7$};
    \draw[black,fill=black] (-2,5) circle (.7ex) node[left] {$\sigma_4$};
    \draw[black,fill=black] (-1,4) circle (.7ex) node[left] {$\sigma_3$};
    \draw[black,fill=black] (0,3) circle (.7ex);
    \draw[black,fill=black] (1,2) circle (.7ex) node[right] {$\sigma_{13}$};
    \draw[black,fill=black] (2,1) circle (.7ex) node[right] {$\sigma_{12}$};
\end{scope}
\end{tikzpicture}
\caption{The posets $P_2^{23567184}$ (left) and $P_2^{34671285}$ (right) have the same number of linear extensions.}
\label{fig:equal_r2}
\end{figure}

Experimental evidence for small values of $m$ suggests that Conjecture~\ref{conj:clusterbehavior} holds for $k$ is large enough.
Figures~\ref{fig:graph3} and~\ref{fig:graph4} show the initial values of the sequence 
$$N^\pi_k=\frac{\left(r_{1+k(m-1),k}^\pi\right)^{1/k}}{k^{m-\pi_m+\pi_1-1}}$$ for each \(\pi\in\mfS_m\) in standard form with \(m=7,8\). 
It is interesting to note that, for a fixed difference $\pi_m-\pi_1=d$ and for a fixed $k$ large enough, 
the value of $r_{1+k(m-1),k}^\pi$ seems to increase when $\pi_1$ increases if $d\ge2$ (consistently with Conjecture~\ref{conj:clusterbehavior}), but it seems to decrease when $\pi_1$ increases if $d=1$.

\begin{figure}[htb]
\centering
\begin{tabular}{cc}
\includegraphics[width=2in]{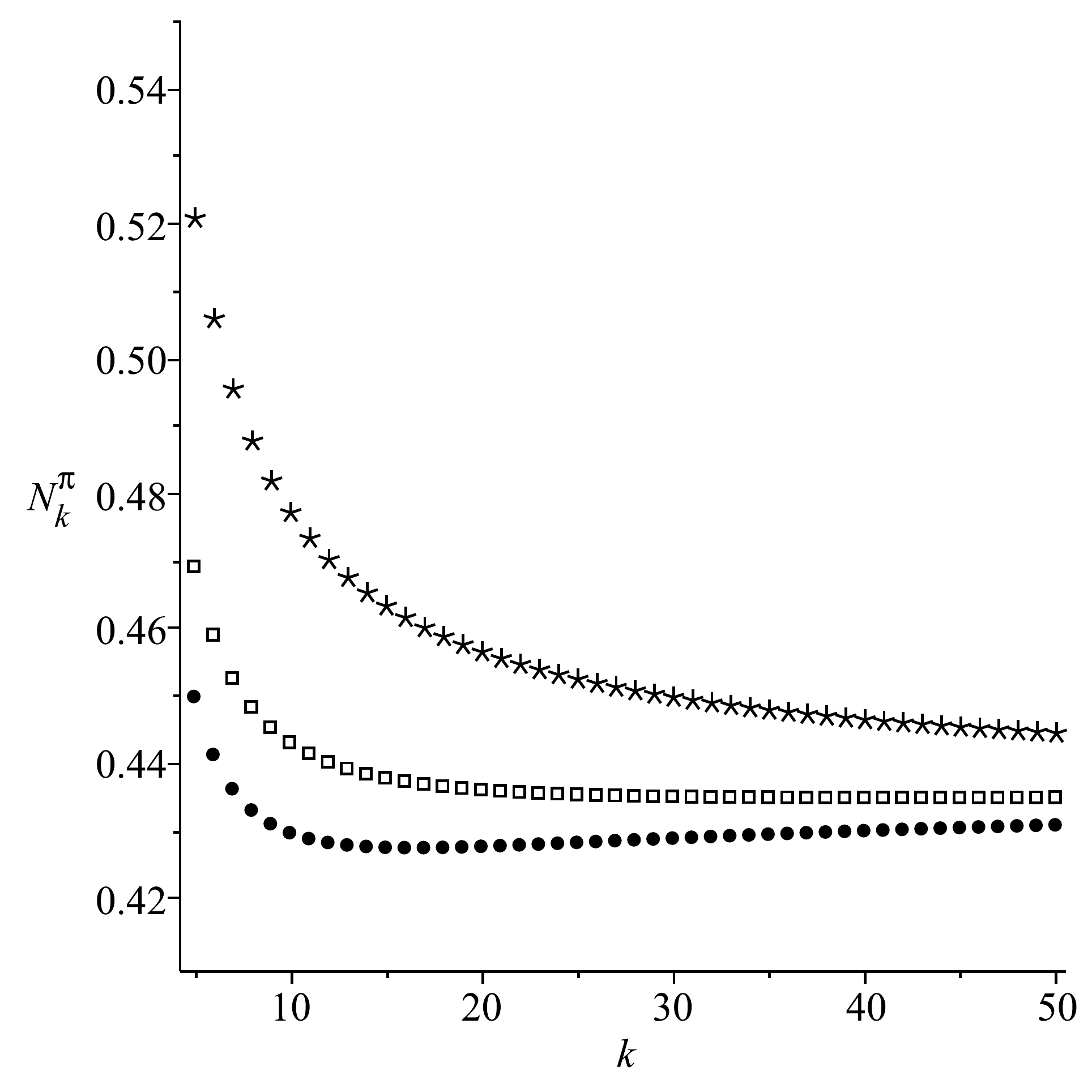} & \includegraphics[width=2in]{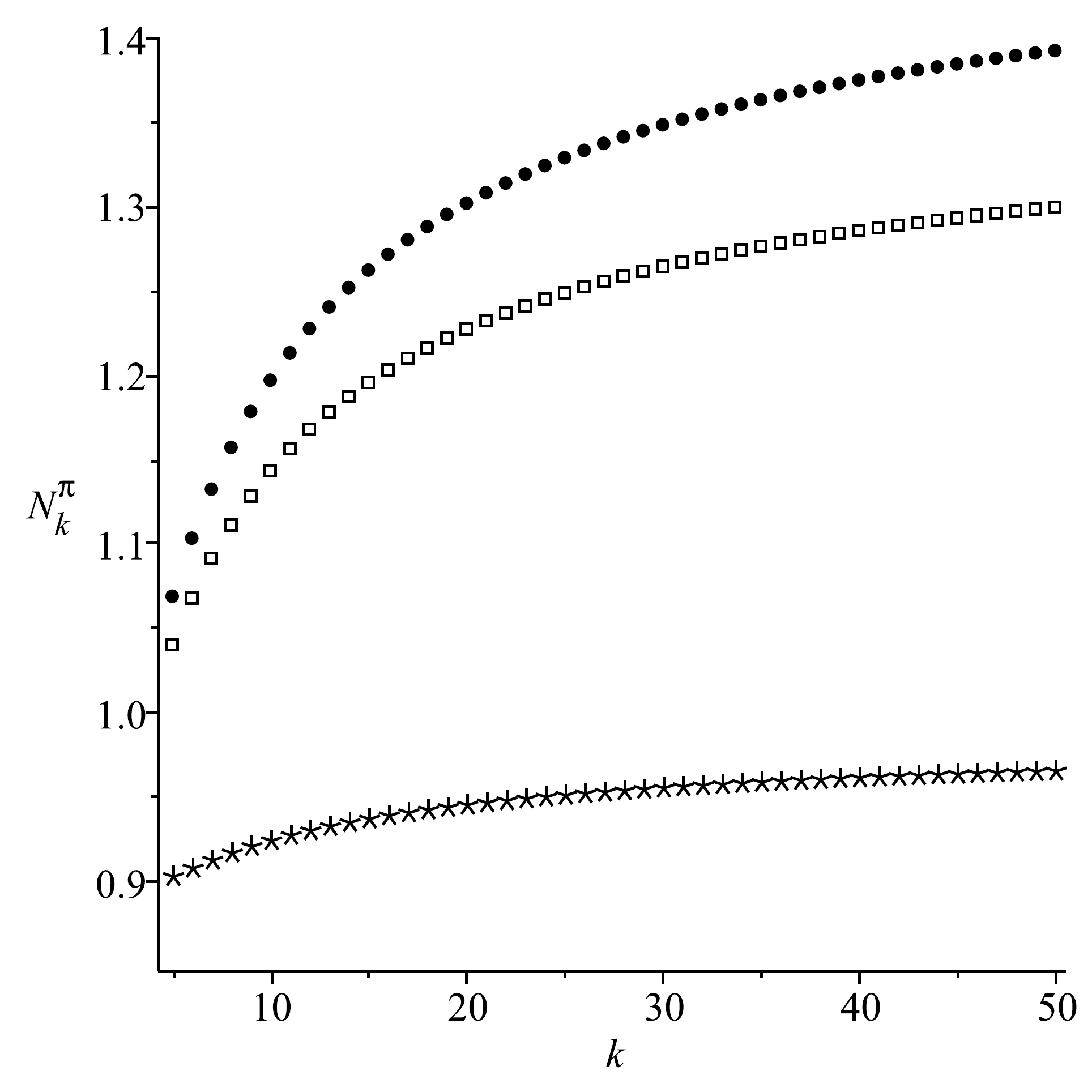}\\
\(\pi_7-\pi_1=1\) & \(\pi_7-\pi_1=2\)\\ \\ 
\includegraphics[width=2in]{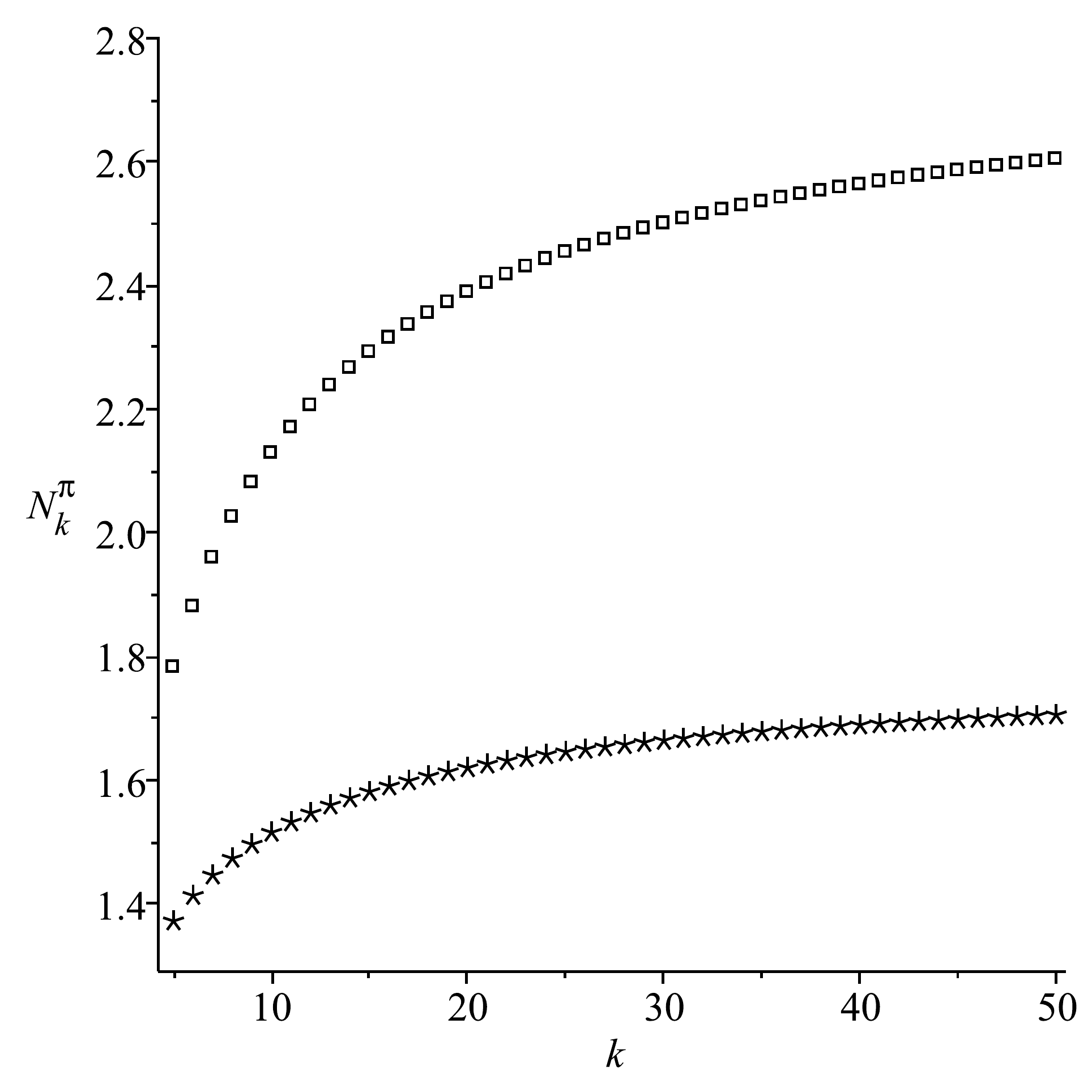} & \includegraphics[width=2in]{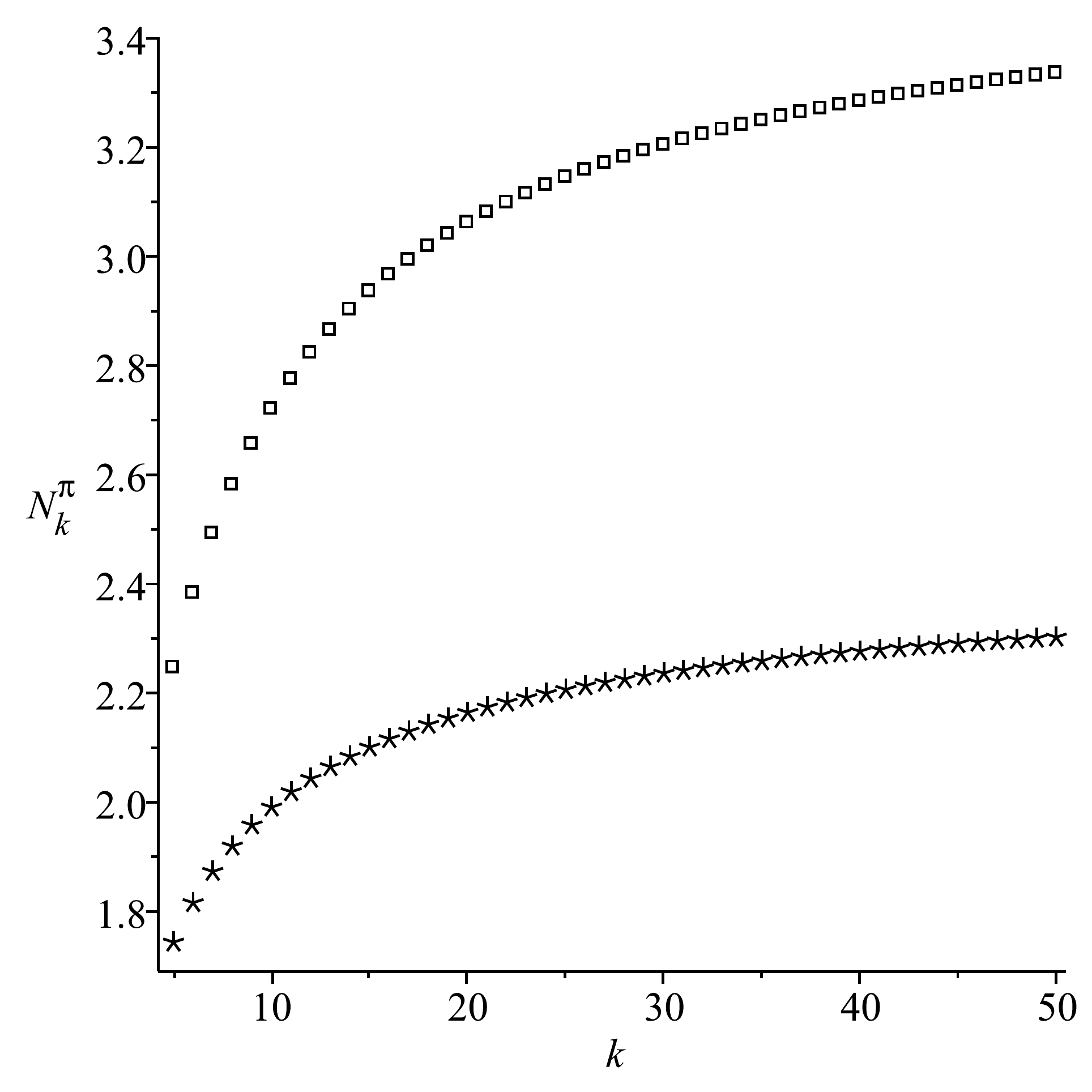}\\
\(\pi_7-\pi_1=3\) & \(\pi_7-\pi_1=4\)
\end{tabular}
\caption{Initial values of $N^\pi_k$ for $\pi\in\mfS_7$.
In each figure, \(\pi_7-\pi_1\) is fixed, and the symbols indicate the value of \(\pi_1\):
$\star\ (\pi_1=1)$, $\square\ (\pi_1=2)$, $\bullet\ (\pi_1=3)$.}
\label{fig:graph3}
\end{figure}

\begin{figure}[htb]
\centering
\begin{tabular}{cc}
\includegraphics[width=2in]{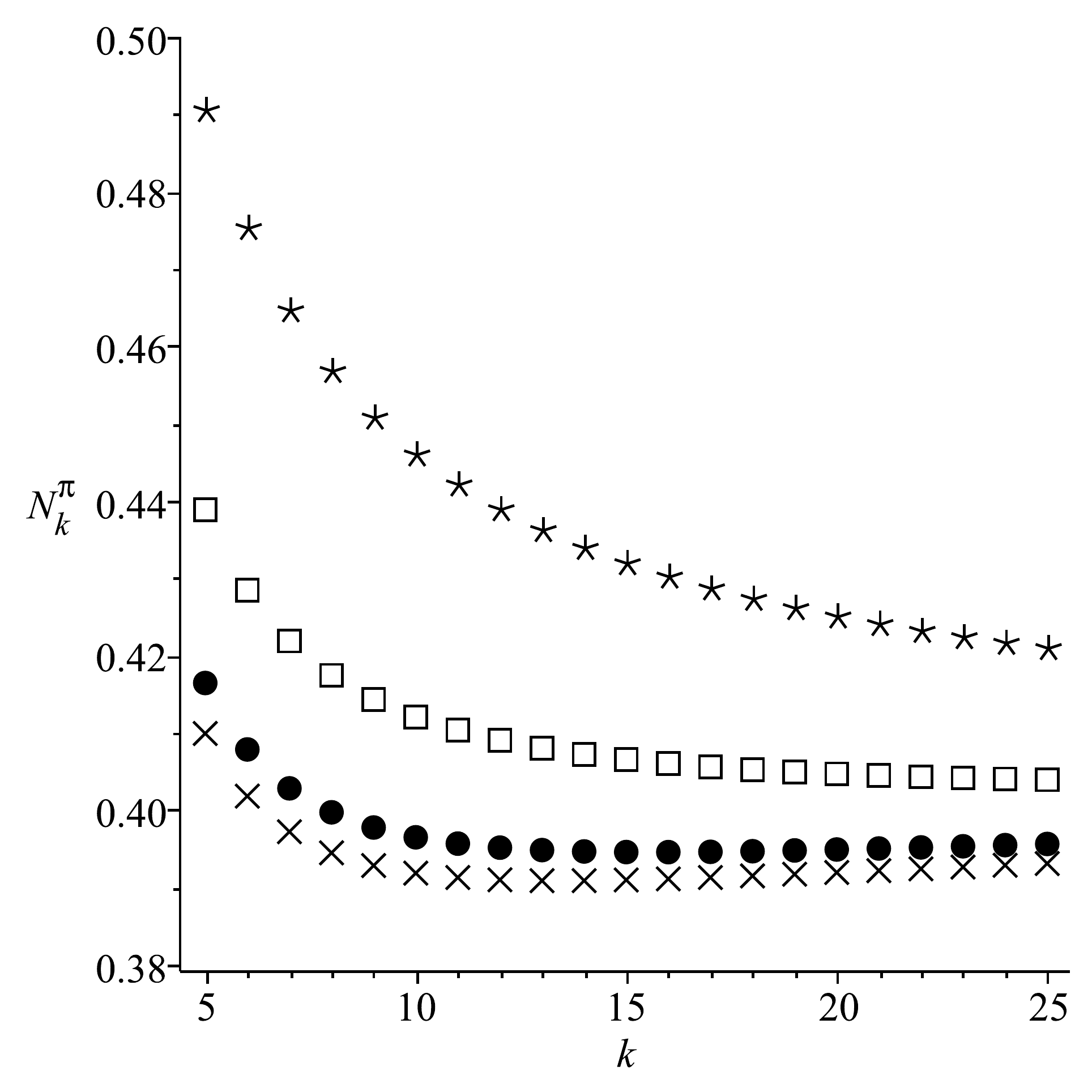} & \includegraphics[width=2in]{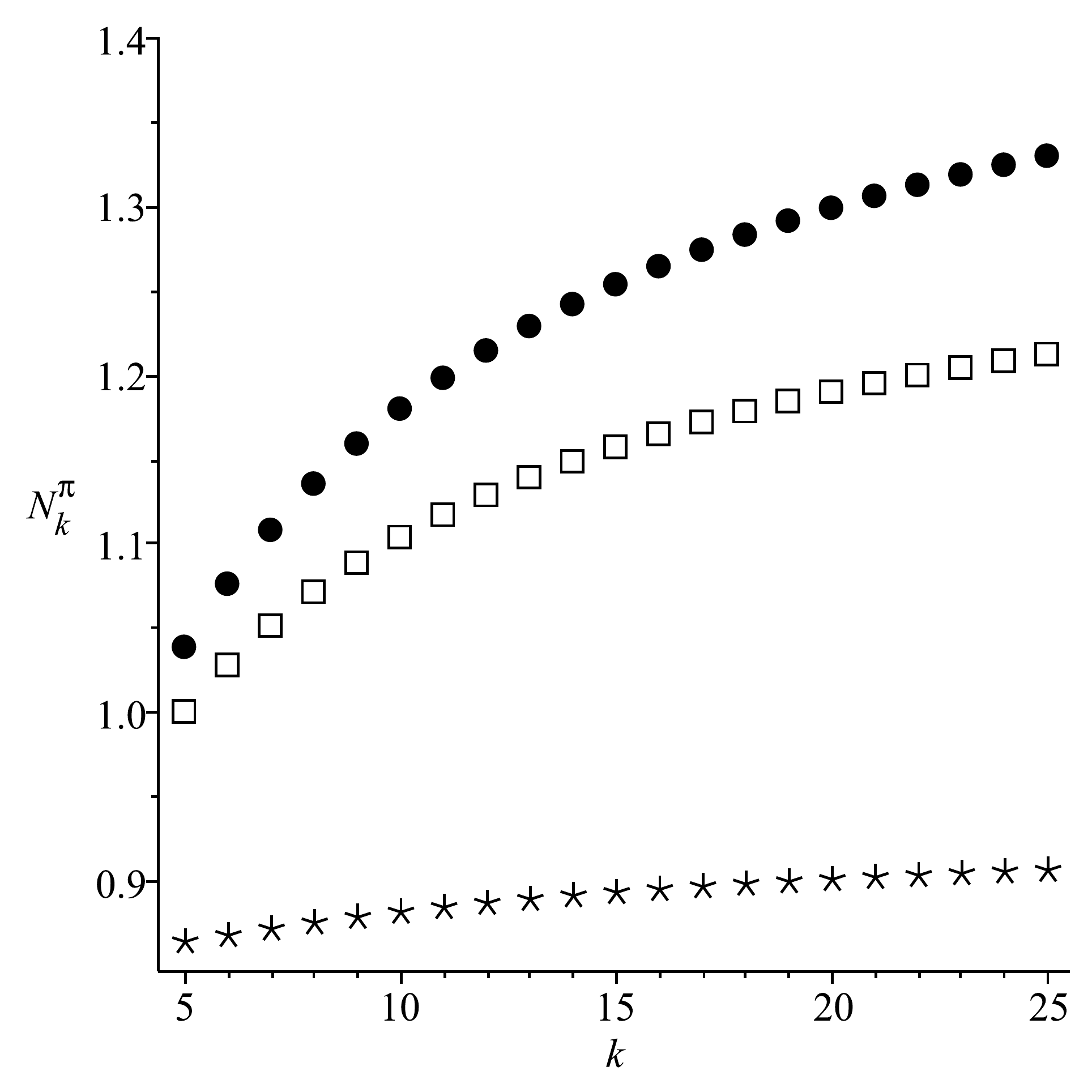}\\
\(\pi_8-\pi_1=1\) & \(\pi_8-\pi_1=2\)
\end{tabular}\bigskip

\begin{tabular}{ccc}
\includegraphics[width=2in]{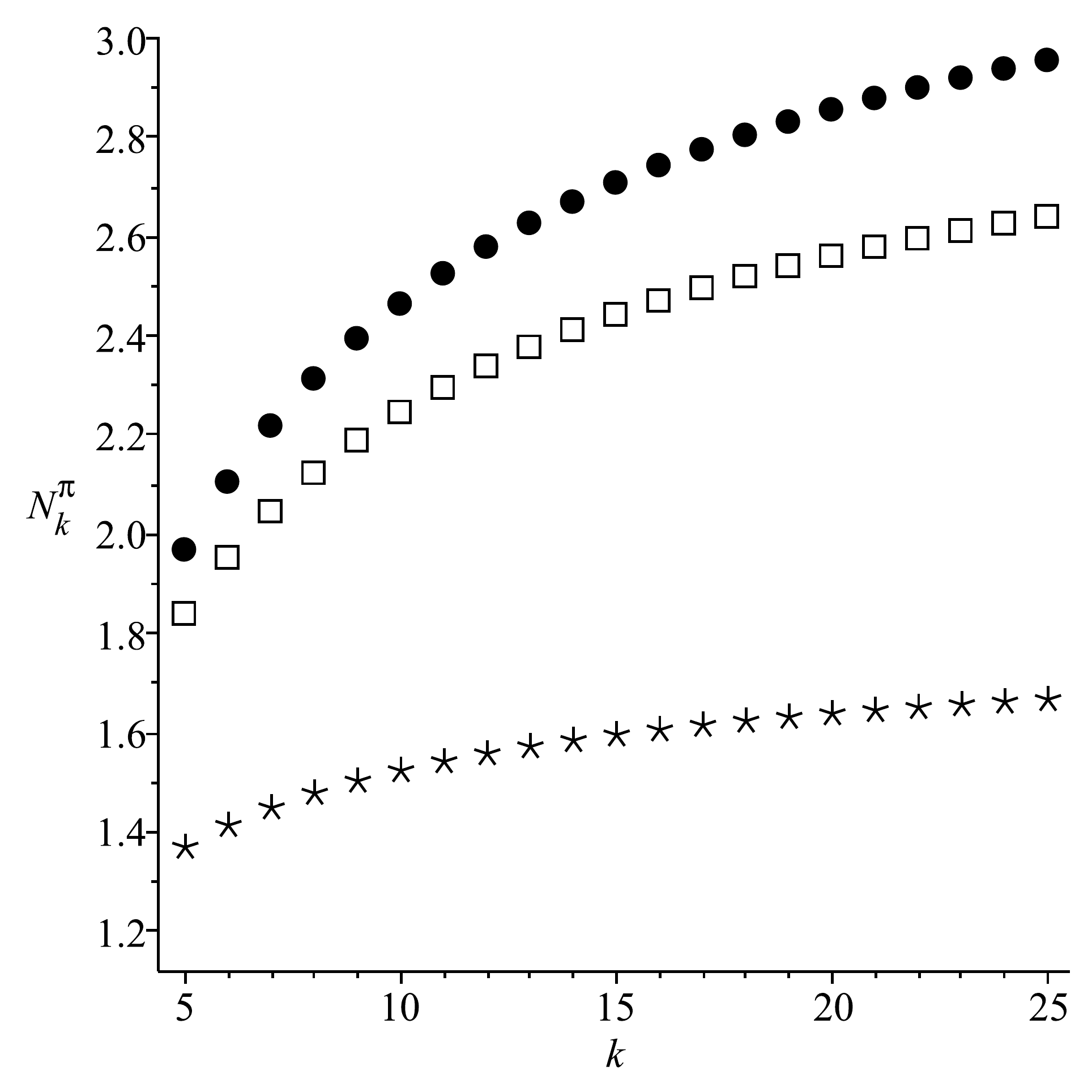} & \includegraphics[width=2in]{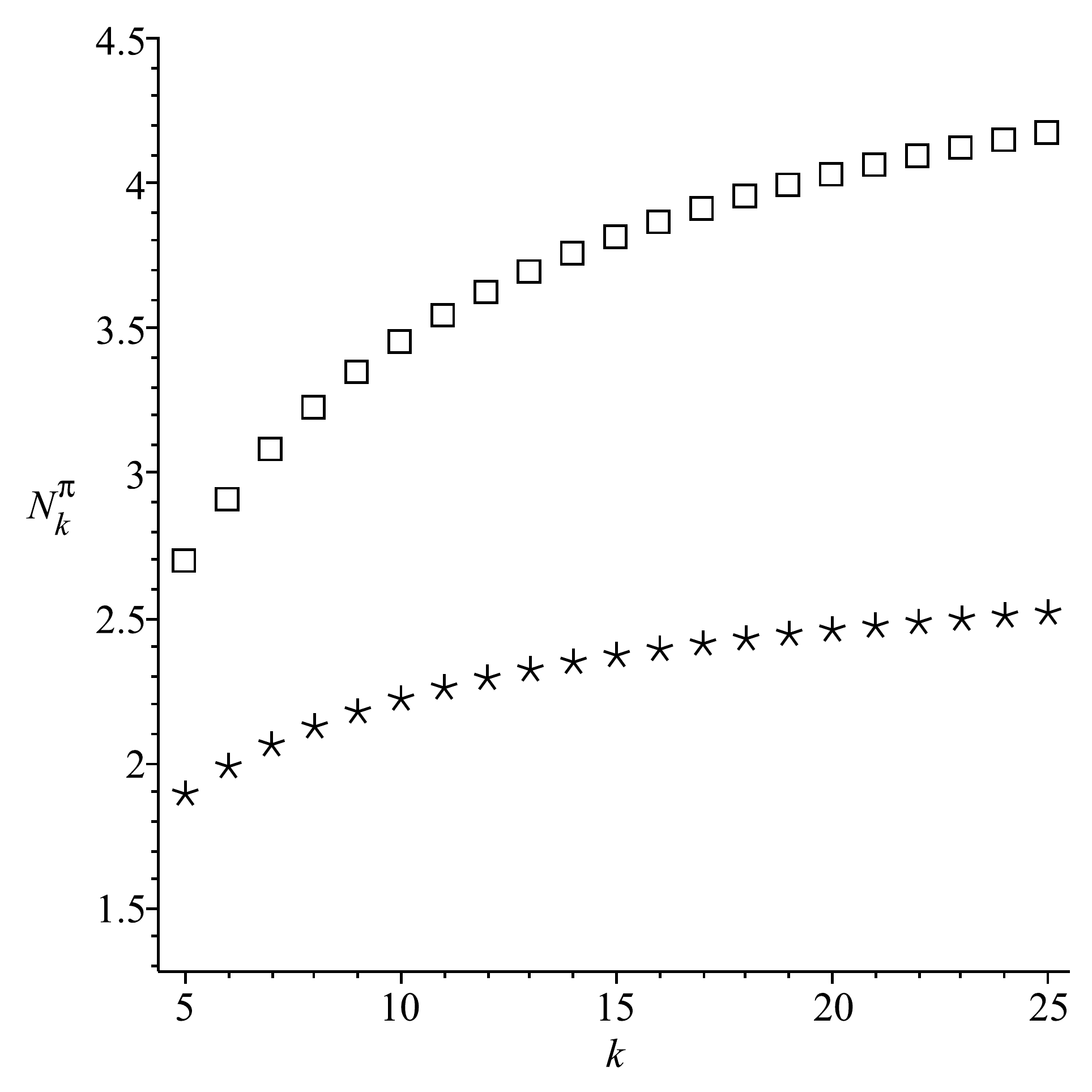} & \includegraphics[width=2in]{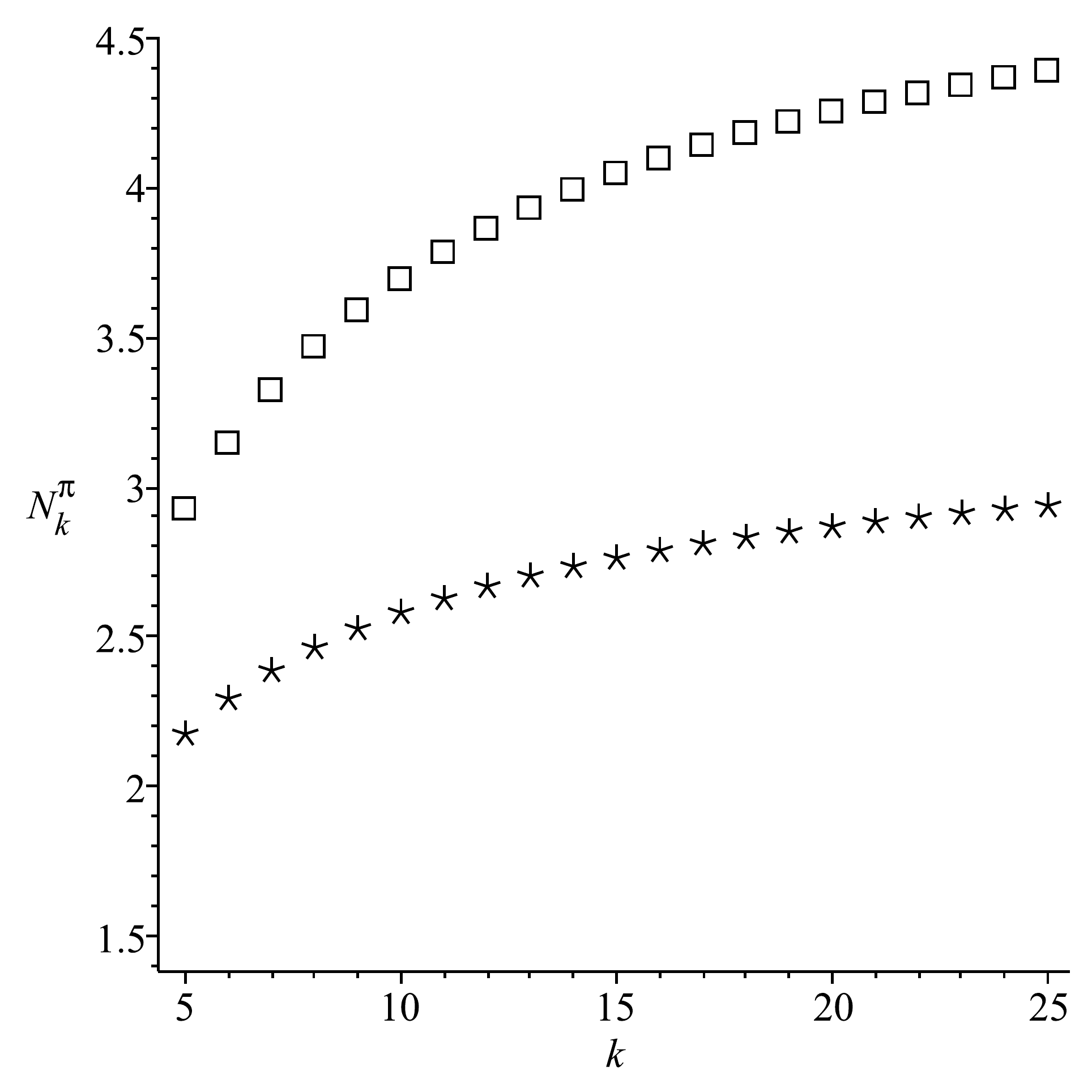}\\
\(\pi_8-\pi_1=3\) & \(\pi_8-\pi_1=4\) & \(\pi_8-\pi_1=5\)
\end{tabular}
\caption{Initial values of $N^\pi_k$ for $\pi\in\mfS_8$.
In each figure, \(\pi_8-\pi_1\) is fixed, and the symbols indicate the value of \(\pi_1\):
$\star\ (\pi_1=1)$, $\square\ (\pi_1=2)$, $\bullet\ (\pi_1=3)$, $\times\ (\pi_1=4)$.}
\label{fig:graph4}
\end{figure}

\section{Super-strong c-Wilf equivalence}\label{superstrong}

Recall from Definition~\ref{superstrong} that $\pi\sswe\tau$ if, for every set $S$, the number of permutations $\sigma\in\mfS_n$ with $\Em(\pi,\sigma)=S$ equals the number of those with $\Em(\tau,\sigma)=S$.

\subsection{Refined cluster numbers}

The refined cluster numbers $r_{n,S}^\pi$, defined in Section~\ref{sec:clusterposets}, can be used to characterize super-strong c-Wilf equivalence in a similar way to how the cluster method (Theorem~\ref{clustermethod}) uses regular cluster numbers to characterize strong c-Wilf equivalence. One difference, however, is that the refined version does not immediately lend itself to a generating function identity.

\begin{prop}
\label{prop:Cluster-ss}
Let $\pi,\tau\in\mfS_m$. Then $\pi\sswe\tau$ if and only if $$r_{n,S}^\pi=r_{n,S}^\tau$$ for all \(n\) and $S$.
\end{prop}

\begin{proof} 
It will be convenient to define \(b_{n,S}^\pi\) (and similarly \(b_{n,S}^\tau\)) to be the number of \(\sigma\in\mfS_n\) with \(S\subseteq\Em(\pi,\sigma)\), that is,
\begin{equation}\label{eq:ba} b_{n,S}^\pi=\sum_{S\subseteq T}a_{n,T}^\pi.\end{equation}
With this definition, $\pi\sswe\tau$ if and only if \(b_{n,S}^\pi=b_{n,S}^\tau\) for all \(n\) and~\(S\).
Indeed, the `only if' direction is clear by Equation~\eqref{eq:ba}, and the `if' direction follows from the principle of inclusion-exclusion.

To prove the forward direction of the theorem, suppose that $\pi\sswe\tau$. Then, by Corollary~\ref{cor:sscoverlap}, $\mcO_\pi=\mcO_\tau$, and so conditions (a) and (b) in Definition~\ref{def:cluster} are the same for both $\pi$ and $\tau$.
If a set $S\subseteq[n-m+1]$ satisfies these conditions, then $$r_{n,S}^\pi=b_{n,S}^\pi=b_{n,S}^\tau=r_{n,S}^\tau,$$ and otherwise $r_{n,S}^\pi=0=r_{n,S}^\tau$.

Next we prove the converse. Suppose that $r_{n,S}^\pi=r_{n,S}^\tau$ for all \(n\) and $S$. It suffices to prove that \(b_{n,S}^\pi=b_{n,S}^\tau\) for all \(n\) and~\(S\). The idea of the proof is to partition $S$ into blocks of overlapping occurrences.

Fix $n$ and $S$. Consider the finest partition of $S$ with the property that if $x,y\in S$ and $|y-x|\le m-1$, then $x$ and $y$ are in the same block. Denote the blocks by $S_1,S_2,\dots,S_q$. For each $i$, let \(m_i=\min S_i\), \(M_i=\max S_i\), and  \(\widehat{S_{i}}=\{j-m_i+1:j\in S_i\}\).
We claim that
\begin{equation}
\label{eqn:probind}
\frac{b_{n,S}^\pi}{n!}=\prod_{i=1}^q \frac{r_{M_i-m_i+m,\widehat{S_i}}^\pi}{(M_i-m_i+m)!}.
 \end{equation}
To prove Equation~\eqref{eqn:probind}, consider a permutation \(\sigma\in\mfS_n\) chosen uniformly at random, and let \(E_i\) be the event \(S_i\subseteq\Em(\pi,\sigma)\). Then the event \(E_1\wedge\dots\wedge E_q\) is equivalent to \(S\subseteq\Em(\pi,\sigma)\), and so it has probability \(b_{n,S}^\pi/n!\). On the other hand, since entries in different blocks of the partition differ by at least $m$, the events \(E_i\) for $i\in[q]$ are mutually independent. Furthermore, \(E_i\) occurs with probability \({r_{M_i-m_i+m,\widehat{S_i}}^\pi}/{(M_i-m_i+m)!}\). Thus, the probability of \(E_1\wedge\dots\wedge E_q\) is given by the right-hand side of Equation~\eqref{eqn:probind}.

Since the refined cluster numbers coincide for $\pi$ and $\tau$, the right-hand side of Equation~\eqref{eqn:probind} stays the same when replacing $\pi$ with $\tau$. It follows that the same holds for the left hand side, and so \(b_{n,S}^\pi=b_{n,S}^\tau\).\end{proof}

\begin{ex} For $\pi=3142$, $S=\{2,4,7,12,14,19,22\}$ and $n=30$, Equation~\eqref{eqn:probind} can be used to express
$b_{n,S}^\pi$ as
$$30! \cdot\frac{r_{9,\{1,3,6\}}^\pi}{9!}\cdot\frac{r_{6,\{1,3\}}^\pi}{6!}\cdot\frac{r_{7,\{1,4\}}^\pi}{7!}.$$
The cluster numbers on the right-hand side are easy to compute when interpreted as counting linear extensions of cluster posets.
\end{ex}

\subsection{A sufficient condition for super-strong c-Wilf equivalence}

Recall that Theorem~\ref{thm:DE-ss} states that, if two permutations satisfy the hypotheses from Theorem \ref{thm:KS}, then they are in fact super-strongly c-Wilf equivalent. We now use Proposition~\ref{prop:Cluster-ss} to prove this result.

\begin{proof}[Proof of Theorem~\ref{thm:DE-ss}]
To prove that \(\pi\sswe\tau\), is it enough by Proposition~\ref{prop:Cluster-ss} to  show that these permutations have the same refined cluster numbers. We will show that, in fact, for every set $S=\{i_1<\dots<i_k\}$ satisfying conditions (a) and (b) in Definition~\ref{def:cluster}, the cluster posets \(P_{n,S}^\pi\) and \(P_{n,S}^\tau\) are isomorphic.

Denote the elements of these posets by $p_1,p_2,\dots,p_n$ and $t_1,t_2,\dots,t_n$, respectively.
Let \(\eta=\pi^{-1}\) and \(\mu=\tau^{-1}\). Recall from Equation~\eqref{eq:chain} that \(P_{n,S}^\pi\) is the transitive closure of the chains
\begin{equation}\label{eq:chain_p}p_{i-1+\eta_1}<\cdots<p_{i-1+\eta_m}
\end{equation}
for each \(i\in S\), and similarly \(P_{n,S}^\tau\) is the transitive closure of the chains
\begin{equation}\label{eq:chain_t}t_{i-1+\mu_1}<\cdots<t_{i-1+\mu_m}.
\end{equation}
Thus, to conclude that  \(P_{n,S}^\pi\) and \(P_{n,S}^\tau\) are isomorphic, it suffices to show that
elements in different chains~\eqref{eq:chain_p} coincide if and only if so do the correponding elements in the chains~\eqref{eq:chain_t}.
More precisely, we have to show that for any $i,i'\in S$ and any $x,y\in[m]$, 
\begin{equation}\label{eq:iff}i-1+\eta_y=i'-1+\eta_x \Longleftrightarrow i-1+\mu_y=i'-1+\mu_x.
\end{equation}
Without loss of generality, we will assume that $i<i'$. By symmetry, if suffices to prove the implication from left to right in~\eqref{eq:iff}.

In the $\pi$-cluster $(p_1p_2\dots p_n,S)$, $p_{i-1+\eta_y}$ is the $\eta_y$-th entry of the occurrence of $\pi$ starting in position $i$, namely the one corresponding to $\pi_{\eta_y}=y$ in the standardized occurrence.  Similarly, $p_{i'-1+\eta_x}$ is the $\eta_x$-th entry of the occurrence of $\pi$ starting in position $i'$, corresponding to $\pi_{\eta_x}=x$ in the standardized occurrence.

Suppose that the left-hand side of~\eqref{eq:iff} holds. This is equivalent to the fact that the entries $p_{i-1+\eta_y}$ and $p_{i'-1+\eta_x}$ are the same. In particular, since the occurrences of $\pi$ in positions $i$ and $i'$ overlap, we have 
$i'-i\in\mcO_\pi=\mcO_\tau$, and so 
\begin{equation}\label{eq:opi}\st(\pi_{i'-i+1}\ldots\pi_m)=\st(\pi_1\ldots\pi_{m-i'+i}),\end{equation}
\begin{equation}\label{eq:otau}\st(\tau_{i'-i+1}\ldots\tau_m)=\st(\tau_1\ldots\tau_{m-i'+i}).\end{equation}

Because of Equation~\eqref{eq:opi}, the rank of $y$ in the set $\{\pi_{i'-i+1},\ldots,\pi_{m}\}$ 
equals the rank of $x$ in the set $\{\pi_1,\ldots, \pi_{m-i'+i}\}$. By hypothesis, these sets equal 
$\{\tau_{i'-i+1},\ldots,\pi_{m}\}$ and $\{\tau_1,\ldots, \tau_{m-i'+i}\}$, respectively. Now, Equation~\eqref{eq:otau} implies
that the position of $y$ in $\tau_{i'-i+1}\ldots\tau_{m}$ equals the position of $x$ in $\tau_1\ldots\tau_{m-i'+i}$. It follows that
$\mu_y-i'+i=\mu_x$, and so the right-hand side of~\eqref{eq:iff} holds.
\end{proof}

\begin{ex}\label{ex:ss}
By Theorem~\ref{thm:DE-ss}, the permutations $\pi=1342675$ and $\tau=1432765$ are super-strongly c-Wilf equivalent because $\mcO_\pi=\{3,6\}=\mcO_\tau$, $\pi_1=1=\tau_1$, $\pi_7=5=\tau_7$, $\{\pi_1,\dots,\pi_4\}=\{1,2,3,4\}= \{\tau_1,\dots,\tau_4\}$,
and $\{\pi_4,\dots,\pi_7\}=\{4,5,6,7\}= \{\tau_4,\dots,\tau_7\}$.

To illustrate the above proof, consider a $\pi$-cluster $(p_1p_2\dots p_{10},S)$ with $S=\{1,4\}$. 
The entry that plays the role of $y=6$ in the occurrence of $\pi$ starting at $i=1$ is $p_{5}$, since $i-1+\eta_y=5$.
The same entry $p_5$ plays the role of $x=3$ in the occurrence starting at $i'=4$, since $i'-1+\eta_x=5$ as well. Now $y=6$ is the second largest element of the set $\{\pi_4,\dots,\pi_7\}=\{4,5,6,7\}= \{\tau_4,\dots,\tau_7\}$, as $x=3$ is of the set $\{\pi_1,\dots,\pi_4\}=\{1,2,3,4\}= \{\tau_1,\dots,\tau_4\}$. By Equation~\eqref{eq:otau}, the position of $y=6$ in $\tau_4\tau_5\tau_6\tau_7=2765$ equals the position of $x=3$ in $\tau_1\tau_2\tau_3\tau_4=1432$, namely the third position. Thus, so $\mu_y-i'+i=\mu_x=3$, or equivalently $i-1+\mu_y=i'-1+\mu_x=6$.
\end{ex}

\subsection{Comparisons among equivalence relations}

It is important to point out that Conjecture \ref{conj:Nak} does not extend to super-strong c-Wilf equivalence, that is, there are permutations that are strongly c-Wilf equivalent but not super-strongly c-Wilf equivalent.
For example, it is easy to compute that $a_{9,\{1,3,6\}}^{1423}=10\neq 6=a_{9,\{1,3,6\}}^{3241}$, 
despite the fact that \(1423^R=3241\). 

However, we have proved that the three equivalence relations that we have defined for consecutive patterns do in fact coincide when restricted to non-overlapping permutations. The following theorem generalizes Lemma~\ref{lem:eq_seq}. Aside from the proof given below, an alternative, less constructive proof can be obtained using Proposition~\ref{prop:Cluster-ss} and some ideas from~\cite{ElizaldeMostLeast2013}.

\begin{thm}\label{thm:eq-sseq}
Let \(\pi,\tau \in\mfS_m\) be non-overlapping permutations. If \(\pi\we\tau\), then \(\pi\sswe\tau\). 
\end{thm}

\begin{proof}
It suffices to prove that for any $n$ and $S$, the number $a_{n,S}^\pi$ is uniquely determined by the sequence $\big\{a_{i,0}^\pi\big\}_{i}$.
Our proof is by induction on $l=\max S$, where we set $l=0$ if $S=\emptyset$. 
The base case $l=0$ is trivial, since $a_{n,\emptyset}^\pi=a_{n,0}^\pi$.  

Now suppose that $l>0$, and let $T=S\setminus\{l\}$. We assume that $n\ge l+m-1$, since otherwise $a_{n,S}^\pi=0$ and we are done. Let
$$
\Sigma=\{\sigma\in{\mfS}_n : \Em(\pi,\st(\sigma_1\sigma_2 \dots\sigma_l))=T,\ \Em(\pi, \st(\sigma_{l+1}\sigma_{l+2}\dots \sigma_n))=\emptyset\}.
$$
We can count the number of permutations in $\Sigma$ by first choosing the values of the $l$ leftmost entries:
\begin{equation}\label{eq:Sigma} |\Sigma|=\binom{n}{l}\, a_{l,T}^\pi\,  a_{n-l,0}^\pi.
\end{equation}

On the other hand, since $\pi$ is non-overlapping, we see that $\sigma\in\Sigma$ if and only if either $\Em(\pi,\sigma)=T$ or $\Em(\pi,\sigma)=T\cup\{j\}$
for some $l-m+2\le j\le l$.  It follows that
$$|\Sigma|=a_{n,T}^\pi+\sum_{j=l-m+2}^l a_{n,T\cup\{j\}}^\pi.$$
Rearranging terms and using Equation~\eqref{eq:Sigma} gives
$$
a_{n,S}^\pi=a_{n,T\cup\{l\}}^\pi=\binom{n}{l}\, a_{l,T}^\pi\,  a_{n-l,0}^\pi-a_{n,T}^\pi-\sum_{j=l-m+2}^{l-1} a_{n,T\cup\{j\}}^\pi.
$$
By the induction hypothesis, the right-hand side is uniquely determined by the sequence $\big\{a_{i,0}^\pi\big\}_{i}$.
\end{proof}

As pointed out to us by Bruce Sagan, the above proof shows that, if $\pi\in\mfS_m$ is non-overlapping and $S$ is a fixed set of positive integers, then we can express
$a_{n,S}^\pi$ (for $n\ge \max S+m-1$) as a polynomial in $n$ of degree $\max S$, where the coefficients belong to the polynomial ring ${\mathbb Q}\big[\big\{a_{k,0}^\pi\big\}_k\big]$.
A similar result for permutations with a given peak set was obtained by Billey, Burdzy and Sagan~\cite{BBS2013}.

Since every permutation is c-Wilf equivalent to its reversal, the following result an immediate consequence of Theorem~\ref{thm:eq-sseq}.
 
\begin{cor}\label{cor:reversal}
If $\pi$ is a non-overlapping permutation, then $\pi\sswe\pi^R$.
\end{cor}

For example, $34671285\sswe58217643$ by the above corollary. 
There seems to be no simple direct combinatorial proof of Corollary~\ref{cor:reversal}, that is, a bijection from $\mfS_n$ to itself that replaces all occurrences of $\pi$ with occurrences of $\pi^R$ without creating additional ones.
On the other hand, one can easily prove bijectively that $b_{n,S}^\pi=b_{n,S}^{\pi^R}$ for all $S$, as defined in Equation~\eqref{eq:ba}, from where the equality $a_{n,S}^\pi=a_{n,S}^{\pi^R}$ follows by inclusion-exclusion.
Indeed, for non-overlapping $\pi\in\mfS_m$ and a fixed $S\subseteq[n-m+1]$, construct a bijection  \(\{\sigma\in\mfS_n\): \(S\subseteq \Em(\pi,\sigma)\}\to \{\sigma\in\mfS_n : S\subseteq \Em(\pi^R,\sigma)\}\) as follows. Partition \(S\) into maximal blocks \(S_1,\ldots, S_q\) of overlapping (marked) occurrences  as in the proof of Proposition \ref{prop:Cluster-ss},
and let \(m_i=\min S_i\) and \(M_i=\max S_i\). The bijection then amounts to replacing each subword $\sigma_{m_i}\sigma_{m_i+1}\dots \sigma_{M_i+m-1}$ in $\sigma$ with its reversal $\sigma_{M_i+m-1}\dots \sigma_{m_i+1}\sigma_{m_i}$, for $1\le i\le q$.

Corollary~\ref{cor:reversal} also follows from Proposition~\ref{prop:Cluster-ss} and the observation that $r^\pi_{n,S}=r^{\pi^R}_{n,S}$ for non-overlapping $\pi$ and any set $S$. This is because, for all $S$ where these refined cluster numbers are non-zero, the corresponding cluster posets $P^\pi_{n,S}$ and $P^{\pi^R}_{n,S}$ are isomorphic.

Note that Theorem~\ref{thm:eq-sseq} and Corollary~\ref{cor:reversal} fail in general for arbitrary permutations; as pointed out earlier, $1423\not\sswe3241$. 
 Even if we require the permutations to be in standard form, Theorem~\ref{thm:eq-sseq} does not generalize. The smallest counterexample is in $\mfS_5$: the permutations $\pi=23514$ and $\tau=25134$ are in standard form and strongly c-Wilf equivalent, since $\tau=\pi^{RC}$. However, they are not super-strongly c-Wilf equivalent, as can be seen by computing $r^{23514}_{12,\{1,4,8\}}=148\neq180= r^{25134}_{12,\{1,4,8\}}$.

As in the case of c-Wilf equivalence and strong c-Wilf equivalence, it is an open problem (although plausibly a more attainable one) to give a simple characterization of super-strong c-Wilf equivalence classes. The sufficient condition for strong c-Wilf equivalence given by Theorem~\ref{thm:DE-ss} is not a necessary one, even if we require the permutations to be in standard form. Indeed, one can check that 
$\pi=123546$ and $\tau=124536$ are super-strongly c-Wilf equivalent, since their cluster posets $P^\pi_{n,S}$ and $P^\tau_{n,S}$ are isomorphic (they are in fact chains). However, as pointed out earlier, $4\in\mcO_\pi=\mcO_\tau$ but $\{\pi_5,\pi_6\}\neq\{\tau_5,\tau_6\}$. Similarly, the pairs $\pi=123645$ and $\tau=124635$;
$\pi=132465$ and $\tau=142365$; and $\pi=154263$ and $\tau=165243$ are super-strongly c-Wilf equivalent, again because their cluster posets are isomorphic, but they do not satisfy the hypotheses of Theorem~\ref{thm:DE-ss}. 

The above examples, along with the simple observation that the cluster posets for $\pi$ and $\pi^C$ are dual of each other,  may lead one to believe that two permutations $\pi$ and $\tau$ are super-strongly c-Wilf equivalent if and only if their
cluster posets $P^\pi_{n,S}$ are $P^\tau_{n,S}$ are isomorphic or dual of each other for every $S$. However, this is not the case in general.
The smallest counterexample is given by $\pi=13425$ and its reversal $\pi^R=52431$. As we will show next, these patterns are super-strongly c-Wilf equivalent, but the posets $P^{\pi}_{12,S}$ and $P^{\pi^R}_{12,S}$ for $S=\{1,4,8\}$ are neither isomorphic nor dual of each other. This
phenomenon is a particular case of Theorem \ref{thm:ssc-SufficientMAXMIN}, which shows that non-isomorphic cluster posets may still have the same number of linear extensions.

\begin{thm}
\label{thm:ssc-SufficientMAXMIN}
Let \(\pi\in\mfS_m\). If \(\pi_1=1, \pi_m=m\), and $|\mcO_\pi|=2$, then \(\pi \sswe \pi^R\).
\end{thm}

\begin{proof}
Since $m-1\in\mcO_\pi$ for every $\pi\in\mfS_m$, we can write $\mcO_\pi=\{x,m-1\}$ for some $x<m-1$.
By Proposition~\ref{prop:Cluster-ss}, it is enough to show that $r^\pi_{n,S}=r^{\pi^R}_{n,S}$ for all $n$ and $S$. 
We will assume that  \(S=\{i_1<i_2<\cdots<i_k\}\) satisfies conditions (a) and (b) in Definition~\ref{def:cluster}, since otherwise $r^\pi_{n,S}=0=r^{\pi^R}_{n,S}$. We prove that $r^\pi_{n,S}=r^{\pi^R}_{n,S}$ by induction on the quantity \(N(S)=|\{j\in[k-1]:i_{j+1}-i_j=m-1\}|\), which counts the number of pairs of marked occurrences that overlap in only one position.

If \(N(S)=0\), then $i_{j+1}-i_j=x$ for all $j\in[k-1]$. In this case, there is a simple bijection between $\pi$-clusters $\{\sigma\in\mfS_n:S\subseteq \Em(\pi,\sigma)\}$ and $\pi^R$-clusters $\{\sigma\in\mfS_n:S\subseteq \Em(\pi^R,\sigma)\}$, given by the reversal map \(\sigma\mapsto\sigma^R\), and so $r^\pi_{n,S}=r^{\pi^R}_{n,S}$ in this case.

For the induction step, suppose that \(N(S)\ge1\), and let $t$ be such that \(i_{t}-i_{t-1}=m-1\). 
Since $\pi_1=1$ and $\pi_m=m$, the value $\sigma_{i_t}$ in a $\pi$-cluster $(\sigma,S)$ is then both a left-to-right maximum and a right-to-left minimum of $\sigma$. Equivalently, $\sigma_1\sigma_2\dots\sigma_{i_t}$ is a permutation of $\{1,2,\dots,i_t\}$, and 
$\sigma_{i_t}\sigma_{i_t+1}\dots\sigma_{n}$ is a permutation of $\{i_t,i_t+1,\dots,n\}$. It follows that, letting
$S_L=\{i_j:1\le j\le t-1\}$ and $S_R=\{i_j-i_{t}+1:t\le j\le k\}$, we have $r^\pi_{n,S}=r^\pi_{i_{t},S_L}r^\pi_{n-i_{t}+1,S_R}$.
A symmetric argument shows that $r^{\pi^R}_{n,S}=r^{\pi^R}_{i_{t},S_L}r^{\pi^R}_{n-i_{t}+1,S_R}$. Since the right-hand sides of these two equalities coincide by the induction hypothesis, we have that $r^\pi_{n,S}=r^{\pi^R}_{n,S}$ as desired.
\end{proof}

Table~\ref{tab:ss} lists strong and super-strong c-Wilf equivalence classes for patterns of length~3, 4 and~5. As shown in~\cite{EN2003,NakamuraClusters}, there are 2 strong
c-Wilf equivalence classes in $\mfS_3$, 7 in $\mfS_4$, and 25 in $\mfS_5$. For length 3, strong and super-strong classes coincide. For length 4,  with the exception of the strong c-Wilf equivalence class $\{1423,4132,2314,3241\}$, which splits into two super-strong classes, all the other  strong c-Wilf equivalence classes are also super-strong classes. This is because they either only contain two elements $\pi$ and $\pi^C$, which are trivially super-strongly c-Wilf equivalent, or because they consist of non-overlapping permutations, which are super-strongly c-Wilf equivalent by Theorem~\ref{thm:eq-sseq}. For length 5, there are 14 strong c-Wilf equivalence classes that split into two super-strong classes. The remaining strong c-Wilf equivalence classes are also super-strong classes, as can be shown using Theorems~\ref{thm:eq-sseq} and~\ref{thm:ssc-SufficientMAXMIN}.

\begin{table}[htb]
$$
\begin{array}{|c|}
\hline
123,321 \\ \hline
\cellcolor{yellow} 132,312, 231,213 \\ \hline
\end{array}
\qquad\qquad
\begin{array}{|c:c|}
\hline
\multicolumn{2}{|c|}{1234,4321} \\ \hline
\multicolumn{2}{|c|}{2413,3142} \\ \hline
\multicolumn{2}{|c|}{2143,3412} \\ \hline
\multicolumn{2}{|c|}{1324,4231} \\ \hline
1423,4132 & 2314,3241 \\ \hline
\multicolumn{2}{|c|}{\cellcolor{yellow} 1342,4213,2431,3124,} \\ 
\multicolumn{2}{|c|}{\cellcolor{yellow}  1432,4123, 2341,3214}  \\ \hline
\multicolumn{2}{|c|}{\cellcolor{yellow} 1243,4312,3421,2134} \\ \hline
\end{array}$$
$$
\begin{array}{|c:c|c:c|c:c|}
\hline
\multicolumn{2}{|c|}{12345, 54321} & 23514, 43152 & 41532, 25134 & \multicolumn{2}{c|}{\cellcolor{yellow} 24153, 42513, 35142, 31524} \\ \cline{1-4}
14253, 52413 & 35241, 31425 & 14523, 52143 & 32541, 34125 & \multicolumn{2}{c|}{\cellcolor{yellow} 25143, 41523, 34152, 32514} \\ \hline
15243, 51423 & 34251, 32415 & 12534, 54132 & 43521, 23145 & \multicolumn{2}{c|}{\cellcolor{green!30} 13425, 53241, 52431, 14235} \\ \hline
24513, 42153 & 31542, 35124 & 15324, 51342 & 42351, 24315 & \multicolumn{2}{c|}{21354, 45312} \\ \hline
13254, 53412 & 45231, 21435 & 21453, 45213 & 35412, 31254 & 15423, 51243 & 32451, 34215 \\ \hline
\multicolumn{2}{|c|}{25314, 41352} & \multicolumn{2}{c|}{\cellcolor{yellow} 13452, 53214, 25431, 41235} & \multicolumn{2}{c|}{\cellcolor{yellow} 21534, 45132, 43512, 23154} \\ \hhline{--~~--} %\cline{1-2} \cline{5-6}
21543, 45123 & 34512, 32154 & \multicolumn{2}{c|}{\cellcolor{yellow} 13542, 53124, 24531, 42135} & \multicolumn{2}{c|}{\cellcolor{yellow} 12453, 54213, 35421, 31245} \\ \cline{1-2}
13524, 53142 & 42531, 24135 & \multicolumn{2}{c|}{\cellcolor{yellow} 14352, 52314, 25341, 41325} & \multicolumn{2}{c|}{\cellcolor{yellow} 12543, 54123, 34521, 32145} \\ \cline{1-2} \cline{5-6}
25413, 41253 & 31452, 35214 & \multicolumn{2}{c|}{\cellcolor{yellow} 14532, 52134, 23541, 43125} & 15234, 51432 & 43251, 23415 \\ \hhline{--~~--} %\cline{1-2} \cline{5-6}
\multicolumn{2}{|c|}{\multirow{2}{*}{14325, 52341}} & \multicolumn{2}{c|}{\cellcolor{yellow} 15342, 51324, 24351, 42315} & \multicolumn{2}{c|}{\cellcolor{green!30} 12435, 54231, 53421, 13245} \\ \hhline{~~~~--} %\cline{5-6}
\multicolumn{2}{|c|}{} & \multicolumn{2}{c|}{\cellcolor{yellow} 15432, 51234, 23451, 43215} &  \multicolumn{2}{c|}{\cellcolor{yellow} 12354, 54312, 45321, 21345} \\ \hline
\end{array}$$
\caption{Strong and super-strong c-Wilf equivalence classes for patterns of length 3, 4 and 5. The dotted lines separate super-strong c-Wilf equivalence classes that are part of the same strong c-Wilf equivalence class. Yellow cells contain non-overlapping patterns, for which Theorem~\ref{thm:eq-sseq} applies; green cells contain classes for which Theorem~\ref{thm:ssc-SufficientMAXMIN} applies.}
\label{tab:ss}
\end{table}

\subsection*{Acknowledgements} 
The authors thank Bruce Sagan and Peter Winkler for useful comments and suggestions.

\printbibliography

\end{document}